\documentclass[12pt]{amsart}
\usepackage{amsmath,amssymb,amsbsy,amsfonts,latexsym,amsopn,amstext,
                                               amsxtra,euscript,amscd,bm}
                    
\usepackage{url}
\usepackage[colorlinks,linkcolor=blue,anchorcolor=blue,citecolor=blue]{hyperref}
\usepackage{color}

\begin{document}
\newtheorem{problem}{Problem}
\newtheorem{theorem}{Theorem}
\newtheorem{lemma}[theorem]{Lemma}
\newtheorem{claim}[theorem]{Claim}
\newtheorem{cor}[theorem]{Corollary}
\newtheorem{prop}[theorem]{Proposition}
\newtheorem{definition}{Definition}
\newtheorem{question}[theorem]{Question}
\newtheorem{conjecture}{Conjecture}

\def\cA{{\mathcal A}}
\def\cB{{\mathcal B}}
\def\cC{{\mathcal C}}
\def\cD{{\mathcal D}}
\def\cE{{\mathcal E}}
\def\cF{{\mathcal F}}
\def\cG{{\mathcal G}}
\def\cH{{\mathcal H}}
\def\cI{{\mathcal I}}
\def\cJ{{\mathcal J}}
\def\cK{{\mathcal K}}
\def\cL{{\mathcal L}}
\def\cM{{\mathcal M}}
\def\cN{{\mathcal N}}
\def\cO{{\mathcal O}}
\def\cP{{\mathcal P}}
\def\cQ{{\mathcal Q}}
\def\cR{{\mathcal R}}
\def\cS{{\mathcal S}}
\def\cT{{\mathcal T}}
\def\cU{{\mathcal U}}
\def\cV{{\mathcal V}}
\def\cW{{\mathcal W}}
\def\cX{{\mathcal X}}
\def\cY{{\mathcal Y}}
\def\cZ{{\mathcal Z}}

\def\A{{\mathbb A}}
\def\B{{\mathbb B}}
\def\C{{\mathbb C}}
\def\D{{\mathbb D}}
\def\E{{\mathbb E}}
\def\F{{\mathbb F}}
\def\G{{\mathbb G}}
\def\I{{\mathbb I}}
\def\J{{\mathbb J}}
\def\K{{\mathbb K}}
\def\L{{\mathbb L}}
\def\M{{\mathbb M}}
\def\N{{\mathbb N}}
\def\O{{\mathbb O}}
\def\P{{\mathbb P}}
\def\Q{{\mathbb Q}}
\def\R{{\mathbb R}}
\def\S{{\mathbb S}}
\def\T{{\mathbb T}}
\def\U{{\mathbb U}}
\def\V{{\mathbb V}}
\def\W{{\mathbb W}}
\def\X{{\mathbb X}}
\def\Y{{\mathbb Y}}
\def\Z{{\mathbb Z}}

\def\ep{{\mathbf{e}}_p}
\def\em{{\mathbf{e}}_m}
\def\eq{{\mathbf{e}}_q}

\def\scr{\scriptstyle}
\def\\{\cr}
\def\({\left(}
\def\){\right)}
\def\[{\left[}
\def\]{\right]}
\def\<{\langle}
\def\>{\rangle}
\def\fl#1{\left\lfloor#1\right\rfloor}
\def\rf#1{\left\lceil#1\right\rceil}
\def\le{\leqslant}
\def\ge{\geqslant}
\def\eps{\varepsilon}
\def\mand{\qquad\mbox{and}\qquad}

\def\sssum{\mathop{\sum\ \sum\ \sum}}
\def\ssum{\mathop{\sum\, \sum}}
\def\ssumw{\mathop{\sum\qquad \sum}}

\def\vec#1{\mathbf{#1}}
\def\inv#1{\overline{#1}}
\def\num#1{\mathrm{num}(#1)}
\def\dist{\mathrm{dist}}

\def\fA{{\mathfrak A}}
\def\fB{{\mathfrak B}}
\def\fC{{\mathfrak C}}
\def\fU{{\mathfrak U}}
\def\fV{{\mathfrak V}}

\newcommand{\bflambda}{{\boldsymbol{\lambda}}}
\newcommand{\bfxi}{{\boldsymbol{\xi}}}
\newcommand{\bfrho}{{\boldsymbol{\rho}}}
\newcommand{\bfnu}{{\boldsymbol{\nu}}}

\def\GL{\mathrm{GL}}
\def\SL{\mathrm{SL}}

\def\Hba{\overline{\cH}_{a,m}}
\def\Hta{\widetilde{\cH}_{a,m}}
\def\Hb1{\overline{\cH}_{m}}
\def\Ht1{\widetilde{\cH}_{m}}

\def\flp#1{{\left\langle#1\right\rangle}_p}
\def\flm#1{{\left\langle#1\right\rangle}_m}
\def\dmod#1#2{\left\|#1\right\|_{#2}}
\def\dmodq#1{\left\|#1\right\|_q}

\def\Zm{\Z/m\Z}

\def\Err{{\mathbf{E}}}

\newcommand{\comm}[1]{\marginpar{%
\vskip-\baselineskip 
\raggedright\footnotesize
\itshape\hrule\smallskip#1\par\smallskip\hrule}}

\def\xxx{\vskip5pt\hrule\vskip5pt}


\title[Large values of Dirichlet polynomials]{Large values of Dirichlet polynomials and zero density estimates for the Riemann zeta function}
 \author[B. Kerr] {Bryce Kerr}
\address{School of Science, The University of New South Wales Canberra, Australia}
\email{b.kerr@adfa.edu.au}
\thanks{The author was supported by Australian Research Council Discovery Project DP160100932.}
\date{\today}
\pagenumbering{arabic}

\begin{abstract}
In this paper we obtain some new estimates for the number of large values of Dirichlet polynomials. Our results imply new zero density estimates for the Riemann zeta function which give a small improvement on results of Bourgain and Jutila.
\end{abstract}

\maketitle
\section{Introduction}
In this paper we consider estimating the number of times a Dirichlet polynomial can take large values. Let $a_n$ be a sequence of complex numbers satisfying $|a_n|\le 1$ and $\cA\subset [0,T]$ a set of real numbers which is $1$-spaced and satisfies
\begin{align*}
\left|\sum_{N\le n \le 2N}a_n n^{it}\right|\ge V, \quad t\in \cA.
\end{align*}
The problem of obtaining an upper bound for the cardinality $|\cA|$ is motivated by estimating the number of zeros of the Riemann zeta function in a rectangle to the right of the critical line. The main conjecture for this problem is known as Montgomery's conjecture and states
\begin{align}
\label{eq:mont}
\sum_{t\in \cA}\left|\sum_{N\le n \le 2N}a_n n^{it}\right|^2\ll (NT)^{o(1)}\left(N|\cA|+N^2\right).
\end{align}
Assuming $N\ge T^{o(1)}$ this implies  
\begin{align}
\label{eq:montlarge}
|\cA|\ll \frac{N^{2+o(1)}}{V^2}, \quad \text{provided} \quad V\ge N^{1/2+o(1)}.
\end{align}
The best known general result towards~\eqref{eq:mont} is obtained via Fourier completion 
\begin{align}
\label{eq:completion}
\sum_{t\in \cA}\left|\sum_{N\le n \le 2N}a_n n^{it}\right|^2\ll (NT)^{o(1)}\left(NT+N^2\right).
\end{align}
The advantage of~\eqref{eq:mont} over~\eqref{eq:completion} is the lack of dependence on the parameter $T$. Results of this type are very rare and we mention an important estimate over convolutions due to Heath-Brown~\cite{HB}.
\begin{theorem}
\label{thm:heathbrown}
Let $\cA\subseteq [0,T]$ be a well spaced set, $N$ an integer and $a_n$ a sequence of complex numbers satisfying $|a_n|\le 1$. We have  
$$\sum_{t_1,t_2\in \cA}\left|\sum_{1\le n \le N}a_n n^{-1/2+i(t_1-t_2)}\right|^2\ll (|\cA|^2+N|\cA|+T^{1/2}|\cA|^{5/4})(NT)^{o(1)}.$$
\end{theorem}
This bound was first used by Heath-Brown to estimate the additive energy $E(\cA)$ of large values 
$$E(\cA)=|\{ t_1,\dots,t_4 \in\cA \ : \ |t_1+t_2-t_3-t_4|\le 1\},$$
and applied to zero density estimates~\cite{HB1} by combining with ideas of Jutila~\cite{Jut} and primes in short intervals~\cite{HB2}. Bourgain~\cite{Bou} refined the use of Theorem~\ref{thm:heathbrown} in applications to zero density estimates and obtained the widest known parameters for which the density conjecture holds.  The use of energy estimates in classical arguments for bounding $|\cA|$ interact badly with Huxley's subdivision and Bourgain was able to refine this aspect of Heath-Brown's argument by exploiting the fact that if $|\cA|$ is large then the points $t\in \cA$ also correlate with large values of the Riemann zeta function. 
\newline

 Bourgain's argument can be considered a Balog-Szemeredi-Gower's type theorem applied to $\cA$. One way to see how ideas from additive combinatorics are useful is as follows. Using some Fourier analysis, we may assume $\cA\subseteq \Z$. Suppose we are in an extreme case where 
\begin{align}
\label{eq:AAsumset}
|\cA-\cA|\ll |\cA|.
\end{align}
Then 
$$E(\cA)\sim |\cA|^3,$$
and hence most points $t\in \cA$ have $\gg |\cA|$ representations 
$$t=t'-t'', \quad t'\in \cA, \quad t''\in \cA-\cA,$$
which implies
\begin{align*}
\sum_{t\in \cA}\left|\sum_{N\le n \le 2N}a_n n^{it}\right|^2\ll \frac{1}{|\cA|}\sum_{t'\in \cA, t''\in \cA-\cA}\left|\sum_{N\le n \le 2N}a_n n^{i(t'-t'')} \right|^2.
\end{align*}
By~\eqref{eq:AAsumset} we can cover $\cA-\cA$ by $O(1)$ translates of $\cA$, so that 
\begin{align*}
\sum_{t\in \cA}\left|\sum_{N\le n \le 2N}a_n n^{it}\right|^2\ll \frac{1}{|\cA|}\sum_{t'\in \cA, t''\in \cA}\left|\sum_{N\le n \le 2N}a'_n n^{i(t'-t'')} \right|^2,
\end{align*}
for some sequence $a'_n$ satisfying $|a'_n|=|a_n|$. Applying Theorem~\ref{thm:heathbrown} gives 
$$\sum_{t\in \cA}\left|\sum_{N\le n \le 2N}a_n n^{it}\right|^2\ll N^{o(1)}(N|\cA|+N^2+T^{1/2}|\cA|^{5/4}),$$
which establishes~\eqref{eq:mont} for certain ranges of parameters. One would then give a complementary approach to deal with the case $E(\cA)$ is small and a final bound for $|\cA|$ may be obtained by decomposing $\cA$ into pieces with either small energy or small sumset. The most straightforward way to deal with small energy is to apply duality to create more variables of summation in $\cA$. Directly applying dualtiy to~\eqref{eq:mont} for $\ell_2$ norms sets a limit of the argument at
\begin{align}
\label{eq:mont11}
\sum_{t\in \cA}\left|\sum_{N\le n \le 2N}a_n n^{it}\right|^2\ll (NT)^{o(1)}\left(N^{3/2}|\cA|+N^2\right),
\end{align}
and would give the estimate~\eqref{eq:montlarge} in the range $V\ge N^{3/4+o(1)}$. 
\newline

Establishing~\eqref{eq:mont11} would be significant and there have been a number of improvements to~\eqref{eq:completion} for $V\ge N^{3/4+o(1)}$. We refer the reader to~\cite{Ivic} for an overview of techniques and results prior to Bourgain's work~\cite{Bou0,Bou,Bou1}. An important problem which has seen no progress is to improve on~\eqref{eq:completion} for $V\le N^{3/4+o(1)}$. One may consider applying duality with fractional exponents although this approach lacks a geometric way to interpret the resulting mean values as in the $\ell_2$ case. Attempting to deal with this issue by decomposing into level sets, one is lead to sums of the form 
\begin{align}
\label{eq:hbsparse}
\sum_{t_1,t_2\in \cA}\left|\sum_{d\in D} d^{i(t_1-t_2)}\right|^2,
\end{align}
where $D\subseteq [N,2N]\cap \Z$ may be sparse. This would require a variant of Theorem~\ref{thm:heathbrown} which is sensitive to the size of $D$. It is not clear what to expect for the sums~\eqref{eq:hbsparse} since one may construct variations which give the trivial bound. If $q$ is prime and $H\subseteq \F_q$ is some multiplicative subgroup, there exists a set $\cA$ of multiplicative characters mod $q$ such that 
$$|H||\cA|\sim q,$$
and 
\begin{align}
\label{eq:hbsparse1}
\sum_{\chi_1,\chi_2\in \cA}\left|\sum_{h\in H} \chi_1\overline \chi_2(h)\right|^2= |H|^2|\cA|^2.
\end{align} 
If $D$ is not small it does not seem possible to construct similar examples directly for the sums~\eqref{eq:hbsparse} since for any set of integers $D\subseteq [N,2N]$ satisfying $|D|\ge N^{\varepsilon}$ we have 
$$|DD|\ge |D|^{2-o(1)}.$$
\indent In this paper we obtain some new large values estimates in the range $V\ge N^{3/4+o(1)}$. Our arguments build on techniques of Bourgain, Heath-Brown, Huxley, Jutila and Ivi\'{c} and are also motivated by the sum-product problem. Current approaches to the sum-product problem establish relations between various energies using geometric incidences. To estimate the number of large values of exponential sums, one may proceed in analogy to sum-product estimates given a suitable replacement for geometric incidences. In the case of Dirichlet polynomials, this role is played by Heath-Brown's convolution estimate.
\newline

We will use a more general version of Theorem~\ref{thm:heathbrown} and in particular we consider the sums 
\begin{align}
\label{eq:hbgeneral}
\sum_{\substack{t_1,t_2\in \cA \\ |t_1-t_2|\le \delta T}}\gamma(t_1)\gamma(t_2)\left|\sum_{1\le n \le N}a_n n^{-1/2+i(t_1-t_2)}\right|^2.
\end{align}
The parameter $\delta$ corresponds to Huxley's subdivision and in applications, estimates for the sums~\eqref{eq:hbgeneral} give a better dependence on $\delta$ than directly using Theorem~\ref{thm:heathbrown}. We note that one may use the sums~\eqref{eq:hbgeneral} in Heath-Brown's original argument~\cite{HB1} to recover Bourgain's result~\cite{Bou}.  We record the bound obtained by this method.
\begin{theorem}
\label{thm:bou}
Suppose $N,T,V$ are positive real numbers and $a_n$ a sequence of complex numbers satisfying $|a_n|\le 1$. Let  $\cA\subset[0,T]$ be a $1$-spaced set satisfying 
\begin{align*}
\left|\sum_{N\le n\le 2N}a_n n^{it} \right|\ge V, \quad t\in \cA.
\end{align*}
Suppose $N,T,\cA,V$ satisfy
\begin{align}
\label{eq:main1cond} 
N\ge T^{2/3}, \quad  |\cA|\le N, \quad V\ge N^{3/4+o(1)}.
\end{align}
For any $0<\delta\le 1$ we have 
\begin{align*}
|\cA|\ll \frac{1}{\delta}\frac{N^{2+o(1)}}{V^2}+\frac{\delta T^2 N^{4+o(1)}}{V^8}+\frac{T^{1/3}N^{16/3+o(1)}}{\delta^{1/3}V^{20/3}}+\frac{T^{2/3}N^{9+o(1)}}{V^{12}}.
\end{align*}
\end{theorem}

The $\gamma$ in~\eqref{eq:hbgeneral} may be arbitrary positive weights and allow estimation of more general energies such as
\begin{align*}
T_k(\cA)=\{ t_1,\dots,t_{2k}\in \cA \ : \ |t_1+\dots-t_{2k}|\le 1\},
\end{align*}
which are relevant to other problems involving the distribution of primes although will not be considered in this paper. 
\newline 

We also introduce another technical refinement into the Bourgain-Heath-Brown approach which is based on an idea of Ivi\'c. The Hal\'{a}sz-Montgomery method to estimate the number of large values reduces to bounding sums of the form
\begin{align*}
S=\sum_{t_1,t_2\in \cA}\left|\sum_{N\le n \le 2N}n^{-1/2+i(t_1-t_2)}\right|^2.
\end{align*}
Using Mellin inversion one may complete these sums on to $\zeta(1/2+it)$ to get
$$S\ll 
\sum_{t_1,t_2\in \cA}\left|\zeta\left(\frac{1}{2}+i(t_1-t_2)\right)\right|^2.$$
An estimate for $S$ may be obtained by combining H\"{o}lder's inequality with moment estimates for $\zeta$ and energy estimates for $\cA$. By the approximate functional equation, if $N\le T^{1/2-\delta}$ is small then the completion step is wasteful and we may obtain sharper results by retaining some information about $N$. For example, after rescaling we get
\begin{align*}
S\ll N^{2\delta }\sum_{t_1,t_2\in \cA}\left|\sum_{N\le n \le 2N}n^{-1/2-\delta+i(t_1-t_2)}\right|^2,
\end{align*}
which may be completed on to $\zeta(1/2+\delta+it)$ where higher moment estimates are available.
\section{Large values of Dirichlet polynomials}
\label{sec:largevalue}
In what follows we refer to 1-spaced sets as well spaced.
\begin{theorem}
\label{thm:main1}
Suppose $N,T,V$ are positive real numbers and $a_n$ a sequence of complex numbers satisfying $|a_n|\le 1$. Let  $\cA\subset[0,T]$ a well spaced set satisfying 
\begin{align*}
\left|\sum_{N\le n\le 2N}a_n n^{it} \right|\ge V, \quad t\in \cA.
\end{align*}
Suppose $N,T,\cA,V$ satisfy
\begin{align}
\label{eq:main1cond} 
N\ge T^{2/3}, \quad  |\cA|\le N, \quad V\ge N^{3/4+o(1)}.
\end{align}
Let $k\ge 2$ be a positive integer and $\delta$ a real number satisfying 
\begin{align}
\label{eq:main1delta}
N^{o(1)}\delta \le \frac{1}{T}\frac{V^{4k}}{N^{3k-1}} , \quad \text{and} \quad  N^{o(1)}\delta \le \frac{N^{1+1/(k-1)}}{T}.
\end{align}
We have 
\begin{align*}
|\cA|\ll \frac{1}{\delta}\frac{N^{2+o(1)}}{V^2}+\frac{T^{1/3}}{\delta^{1/3}}\frac{N^{k+4/3+o(1)}}{V^{4k/3+4/3}}.
\end{align*}
\end{theorem}

\begin{theorem}
\label{thm:main4}
Suppose $N,T,V$ are positive real numbers and $a_n$ a sequence of complex numbers satisfying $|a_n|\le 1$. Let  $\cA\subset[0,T]$ be a well spaced set satisfying 
\begin{align*}
\left|\sum_{N\le n\le 2N}a_n n^{it} \right|\ge V, \quad t\in \cA.
\end{align*}
Suppose $T,V,N,\cA$ satisfy
\begin{align}
\label{eq:main4ass}
\quad |\cA|\le \min\left\{N,\frac{N^4}{T^2}\right\} \quad V\ge N^{25/32+o(1)}.
\end{align}
For any $0<\delta\le 1$ satisfying 
\begin{align}
\label{eq:main4deltaass}
\delta \ge \frac{N^{26+o(1)}}{V^{32}T}, \quad N^{o(1)}\delta \le \frac{V^{16}}{N^{11}T},
\end{align} 
we have 
\begin{align*}
|\cA|&\ll\frac{1}{\delta}\frac{N^{2+o(1)}}{V^2}+\frac{\delta T^2 N^{4+o(1)}}{V^{8}}+\frac{N^{8+o(1)}}{\delta^2 TV^8}+\frac{N^{10+o(1)}}{\delta^{2/3}V^{12}}.
\end{align*}
\end{theorem}
Our next result may be used to recover a zero density estimate of Ivi\'{c}~\cite[Theorem~11.5]{Ivic} with a slightly smaller range of parameters.
\begin{theorem}
\label{thm:main12}
Suppose $N,T,V$ are positive real numbers and $a_n$ a sequence of complex numbers satisfying $|a_n|\le 1$. Let  $\cA\subset[0,T]$ be a well spaced set satisfying 
\begin{align*}
\left|\sum_{N\le n\le 2N}a_n n^{it} \right|\ge V, \quad t\in \cA.
\end{align*}
Suppose $T,N,\cA$ satisfy 
$$N\ge T^{2/3}, \quad |\cA|\le N.$$
For any $0<\delta<1$ satisfying 
\begin{align*}
N^{o(1)}\delta \le \frac{V^8}{TN^5},
\end{align*} 
we have 
\begin{align*}
|\cA|\ll \frac{1}{\delta}\frac{N^{2+o(1)}}{V^2}+\frac{\delta^{2/3} T^{4/3}N^{23/3+o(1)}}{V^{12}}+\frac{T^{2/3}N^{14/3+o(1)}}{V^{20/3}}.
\end{align*}
\end{theorem}

In applications to zero density estimates, one may use results of Huxley~\cite{Hux} or Jutila~\cite{Jut} to verify the conditions  on $\cA$ hold in the above results. 
\section{Zero density estimates for the Riemann zeta function}
\label{sec:zerodensity}
For $1/2\le \sigma \le 1$ and $T\gg 1$ we let  $N(\sigma,T)$ count the number of $\rho=\beta+i\gamma$ satisfying
$$\zeta(\rho)=0, \quad \beta \ge \sigma, \quad 0\le \gamma \le T.$$
By combining the results from Section~\ref{sec:largevalue} with the method of zero detection polynomials we give some new bounds for $N(\sigma,T)$. 
\begin{theorem}
\label{thm:zerodensity2}
If $\sigma$ satisfies 
\begin{align}
\label{eq:sigmacond3}
\sigma \ge \frac{23}{29},
\end{align}
then we have 
\begin{align}
\label{eq:bouest}
N(\sigma,T)\ll T^{3(1-\sigma)/2\sigma+o(1)}.
\end{align}
\end{theorem}
Theorem~\ref{thm:zerodensity2} improves a result of Bourgain~\cite{Bou1} who previously obtained the estimate~\eqref{eq:bouest} in the range 
\begin{align*}
\sigma\ge \frac{3734}{4694}=\frac{23}{29}+\frac{162}{68063}.
\end{align*} 
\begin{theorem}
\label{thm:zerodensity1}
If $\sigma$ satisfies 
\begin{align}
\label{eq:density1cond}
\frac{127}{168}\le \sigma \le \frac{107}{138},
\end{align}
then we have 
 \begin{align*}
N(\sigma,T)\ll T^{36(1-\sigma)/(138\sigma-89)+o(1)}+T^{(114\sigma-79)/(138\sigma-89)+o(1)}.
\end{align*}
In particular, if 
$$\frac{127}{168}\le \sigma \le \frac{23}{30},$$
then we have 
 \begin{align}
\label{eq:thm:zerodensity1}
N(\sigma,T)\ll T^{36(1-\sigma)/(138\sigma-89)+o(1)}.
\end{align}
\end{theorem}
 Ivi\'{c}~\cite[Equation~(11.85)]{Ivic} has obtained 
\begin{align*}
N(\sigma,T)\ll T^{3(1-\sigma)/(7\sigma-4)+o(1)}, \quad \frac{3}{4}\le \sigma\le \frac{10}{13}.
\end{align*}
Theorem~\ref{thm:zerodensity1} provides an improvement on this bound in the range  
$$\frac{41}{54}\le \sigma \le \frac{845+\sqrt{7429}}{1212}.$$ Arguments of Jutila~\cite{Jut}, see also~\cite[Section~11.7]{Ivic} imply the estimate 
\begin{align}
\label{eq:zerodensityJJ}
N(\sigma,T)\ll T^{3k(1-\sigma)/((3k-2)\sigma+2-k)+o(1)},
\end{align}
for any integer $k\ge 2$ valid in the range of parameters given by~\cite[Equation~11.76]{Ivic}. Considering the discussion on~\cite[pp. 289]{Ivic}, in order to use~\eqref{eq:zerodensityJJ} for $\sigma\le 13/17$ we need to take $k\ge 5$. Comparing~\eqref{eq:thm:zerodensity1} with $k\ge 5$ of~\eqref{eq:zerodensityJJ}, we obtain an improvement provided
\begin{align*}
\frac{409}{534}\le \sigma\le  \frac{23}{30}=\frac{409}{534}+\frac{1}{1335}.
\end{align*}

\section{Preliminary results}
In what follows we  assume $T^{o(1)}\le N\le T^{O(1)}$. This implies terms $N^{o(1)}$ and $T^{o(1)}$ have the same meaning.
\newline

\label{sec:prelim}
 Given $\Delta>0$ define 
\begin{align}
\label{eq:IdAset}
I(\Delta,\cA)=\sum_{\substack{t_1,t_2\in \cA \\ |t_1-t_2|\le \Delta}}1.
\end{align}

\begin{lemma}
\label{lem:ell2counting}
Let $\cA\subseteq \R$ be finite and $\Delta>0$. For integer $k$ define
$$\cA_{k}=\{ t\in \cA \ : \ k\Delta<t\le (k+1)\Delta\}.$$
Then we have 
\begin{align*}
\sum_{k}|\cA_k|^2\ll I(\Delta,\cA)\ll \sum_{k}|\cA_k|^2.
\end{align*}
\end{lemma}
\begin{proof}
The inequality 
\begin{align*}
\sum_{k}|\cA_k|^2\ll I(\Delta,\cA),
\end{align*}
follows from the observation that if $t_1,t_2\in \cA_k$ then $|t_1-t_2|\le \Delta.$ If $t_1,t_2$ satisfy
$$|t_1-t_2|\le \Delta,$$
then there exists $k_1,k_2$ satisfying
$$|k_1-k_2|\le 1,$$
such that $t_1\in \cA_{k_1}$ and $t_2\in \cA_{k_2}.$
This implies
$$I(\Delta,\cA)\le \sum_{\substack{k_1,k_2 \\ |k_1-k_2|\le 1}}|\cA_{k_1}||\cA_{k_2}|\ll \sum_{k}|\cA_k|^2.$$
\end{proof}

Combining Lemma~\ref{lem:ell2counting} with the Cauchy-Schwarz inequality gives the following result.
\begin{cor}
\label{eq:ell2ell2c}
Let $\cA\subseteq [0,T]$ be a well spaced set and $0<\delta\le 1$. We have 
\begin{align*}
|\cA|^2\ll \frac{1}{\delta} I(\delta T,\cA).
\end{align*}
\end{cor}

For a proof of the following, see~\cite[Theorem~9.4]{IwKo}
\begin{lemma}
\label{lem:classicalmv}
For any set $\cA\subseteq [0,T]$ of well spaced points, integer $N$ and sequence of complex numbers $a_n$ we have 
$$\sum_{t\in \cA}\left|\sum_{1\le n \le N}a_n n^{it}\right|^2\ll (T+N)\|a\|_2^2(\log{N}).$$
\end{lemma}
As a consequence of the above, we have.
\begin{lemma}
\label{lem:classicalmoments}
For any set $\cA\subseteq [0,T]$ of well spaced points, integers $N,k$ and sequence of complex numbers $a_n$ satisfying $|a_n|\le 1$ we have 
$$\sum_{t\in \cA}\left|\sum_{1\le n \le N}a_n n^{-1/2+it}\right|^{2k}\ll (T+N^k)N^{o(1)}.$$
\end{lemma}
For a proof of the following, see~\cite[Theorem~8.4]{Ivic}.
\begin{theorem}
\label{lem:8thmoment}
For $T\gg 1$ we have 
\begin{align*}
\int_{0}^{T}\left|\zeta\left(\frac{5}{8}+it\right)\right|^{8}dt\ll T^{1+o(1)}.
\end{align*}
\end{theorem}
For a proof of the following, see~\cite[Lemma~4.48]{Bou}.
\begin{lemma}
\label{lem:removemax}
For any $N\ge 1, t\in \R$  and sequence of complex numbers $a_n$ satisfying $|a_n|\le 1$ we have 
\begin{align*}
\sum_{N\le n\le 2N}a_n n^{it}\ll \log{N}\int_{|\tau|\le \log{N}}\left|\sum_{N\le n \le 2N}a_n n^{i(t+\tau)}\right|d\tau.
\end{align*}
\end{lemma}
The following is essentially due to Heath-Brown~\cite{HB2}. Since our statement is more general we provide details of the proof.
\begin{lemma}
\label{lem:e2energy}
Let $N,T$ be positive real numbers and $a_n$ a sequence of complex numbers satisfying $|a_n|\le 1.$ Let $\cA\subset [0,T]$ be a well spaced set satisfying
$$\left|\sum_{N\le n \le 2N}a_n n^{it} \right|\ge V, \quad t\in \cA,$$
and  for integer $\ell$ define
$$r(\ell)=|\{ (t_1,t_2)\in \cA \ : \ 0\le t_1-t_2-\ell<1\}|.$$
If $N,\cA,T$ satisfy
\begin{align}
\label{eq:e2conds}
N\ge T^{2/3}, \quad |\cA|\le N,
\end{align}
 then for any set $\cB$ we have 
\begin{align*}
\sum_{\ell\in \cB}r(\ell)^2\ll \frac{N^{3/2+o(1)}}{V^2}|\cA|^{1/2}\sum_{\ell \in \cB}r(\ell)+\frac{N^{4+o(1)}}{V^4}|\cA|.
\end{align*}
\end{lemma}
\begin{proof}
Let 
$$W=\sum_{\ell\in \cB}r(\ell)^2,$$
and for integer $j\ll \log{T}$ define
$$Y_j=\{ \ell\in \cB \ : \  2^{j}\le r(\ell)<2^{j+1}\},$$
so that 
$$\sum_{|j|\ll \log{T}}2^{2j}|Y_j|\ll E(\cA)\ll \sum_{j\ll \log{T}}2^{2j}|Y_j|.$$
By the pigeonhole principle, there exists some $j_0$ such that defining
\begin{align}
\label{eq:deltaDdefdef}
D=Y_{j_0}, \quad 2^{j_0}=\Delta,
\end{align}
we have 
\begin{align}
\label{eq:EcaUL}
|D|\Delta^2\ll W\ll (\log{T})|D|\Delta^2.
\end{align}
Consider the sum
\begin{align*}
S=\int_{|\tau|\le 2\log{N}}\sum_{\substack{\ell\in D, t\in \cA \\ }}\left|\sum_{N\le n \le 2N}a_n n^{-i(\ell+t+\tau)}\right|^2d\tau.
\end{align*} 
Define the set $\cB_0$ by 
$$\cB_0=\{ (\ell,t)\in D\times \cA \ : \exists \ t'\in \cA \ \text{such that} \ |t'-t-\ell|\le 2\},$$
so that 
\begin{align}
\label{eq:SubB}
S\ge \sum_{\substack{(\ell,t)\in \cB_0 }}\int_{|\tau|\le 2\log{N}}\left|\sum_{N\le n \le 2N}a_n n^{i(\ell+t+\tau)}\right|^2d\tau.
\end{align}
If $(\ell,t)\in \cB_0$ then there exists some $t'\in \cA$ and some $|\theta|\le 2$ such that 
$$t'+\theta=\ell+t,$$
which gives 
\begin{align*}
\int_{|\tau|\le 2\log{N}}\left|\sum_{N\le n \le 2N}a_n n^{i(\ell+t+\tau)}\right|^2d\tau=\int_{|\tau|\le 2\log{N}}\left|\sum_{N\le n \le 2N}a_n n^{i(t'+\theta+\tau)}\right|^2d\tau.
\end{align*}
Since $|\theta|\le 2$ we have 
$$(\theta+[-2\log{N},2\log{N}])\cap [-\log{N},\log{N}]=[-\log{N},\log{N}],$$
and hence 
\begin{align*}
\int_{|\tau|\le 2\log{N}}\left|\sum_{N\le n \le 2N}a_n n^{i(\ell+t+\tau)}\right|^2d\tau\ge \int_{|\tau|\le \log{N}}\left|\sum_{N\le n \le 2N}a_n n^{i(t'+\tau)}\right|^2d\tau.
\end{align*}
Applying Lemma~\ref{lem:removemax}
\begin{align*}
\int_{|\tau|\le 2\log{N}}\left|\sum_{N\le n \le 2N}a_n n^{i(\ell+t+\tau)}\right|^2d\tau\gg \frac{V^2}{\log{N}},
\end{align*}
which implies
\begin{align*}
N^{o(1)}S\gg V^2|\cB_0|.
\end{align*}
Recalling~\eqref{eq:deltaDdefdef} and the definition of $Y_j$
\begin{align*}
\Delta |D|\ll \sum_{\ell \in D}r(\ell)\ll |\{ (\ell,t,t')\in D \times \cA \times \cA \ : \ 0<t'-t-\ell\le 1\}|\le |\cB_0|,
\end{align*}
so that
\begin{align}
\label{eq:SdeltaDLB}
\Delta |D|\ll \frac{N^{o(1)}}{V^2}S.
\end{align}
Taking a maximum over $\tau$ in $S$, there exists some sequence of complex numbers $b(n)$ satisfying $|b(n)|\le 1$ such that 
\begin{align*}
S&\ll N^{o(1)}\sum_{\ell\in D, t\in \cA}\left|\sum_{N\le n \le 2N}b(n)n^{i(\ell+t)}\right|^2 \\ 
&\le  \sum_{N\le n_1,n_2\le 2N}\left|\sum_{\ell\in D}\left(\frac{n_1}{n_2}\right)^{i\ell} \right|\left|\sum_{t\in \cA}\left(\frac{n_1}{n_2}\right)^{it}\right|.
\end{align*}
By the Cauchy-Schwarz inequality
\begin{align}
\label{eq:SS1S2}
S^2\ll N^{2+o(1)}S_1S_2,
\end{align}
where 
\begin{align*}
S_1=\sum_{N\le n_1,n_2\le 2N}(n_1n_2)^{-1/2}\left|\sum_{\ell\in D}\left(\frac{n_1}{n_2}\right)^{i\ell} \right|^2,
\end{align*}
and 
\begin{align*}
S_2=\sum_{N\le n_1,n_2\le 2N}(n_1n_2)^{-1/2}\left|\sum_{t\in \cA}\left(\frac{n_1}{n_2}\right)^{it} \right|^2.
\end{align*}
Interchanging summation, we have 
\begin{align*}
S_1\ll \sum_{\ell_1,\ell_2\in D}\left|\sum_{N\le n \le 2N}n^{-1/2+i(\ell_1-\ell_2)}\right|^2.
\end{align*}
and hence by Theorem~\ref{thm:heathbrown} and~\eqref{eq:e2conds}  
\begin{align*}
S_1\ll (|D|^2+N|D|)N^{o(1)}.
\end{align*}
By a similar argument and the assumption $|\cA|\le N$
\begin{align*}
S_2\ll N^{1+o(1)}|\cA|.
\end{align*}
The above estimates combine to give 
\begin{align*}
S\ll N^{3/2+o(1)}|\cA|\left(|D|+N^{1/2}|D|^{1/2}\right),
\end{align*}
and hence 
\begin{align*}
\Delta |D|\ll \frac{N^{3/2+o(1)}}{V^2}|\cA|^{1/2}\left(|D|+N^{1/2}|D|^{1/2}\right).
\end{align*}
This implies either 
\begin{align*}
\Delta \ll \frac{N^{3/2+o(1)}}{V^2}|\cA|^{1/2},
\end{align*}
or 
\begin{align*}
\Delta^2|D|\ll \frac{N^{4+o(1)}}{V^4}|\cA|,
\end{align*}
and the result follows from~\eqref{eq:EcaUL} and the estimate 
\begin{align*}
\Delta |D|\ll \sum_{\ell \in \cB}r(\ell).
\end{align*}
\end{proof}
The following is due to Huxley~\cite{Hux}.
\begin{lemma}
\label{lem:hux}
Suppose $N,T,V$ are positive real numbers and $a_n$ a sequence of complex numbers satisfying $|a_n|\le 1$. Let  $\cA\subset[0,T]$ be a well spaced set satisfying 
\begin{align*}
\left|\sum_{N\le n\le 2N}a_n n^{it} \right|\ge V, \quad t\in \cA.
\end{align*}
If $V\ge N^{3/4+o(1)}$ then
\begin{align*}
|\cA|&\ll \frac{N^{2+o(1)}}{V^2}+\frac{TN^{4+o(1)}}{V^6}.
\end{align*}
\end{lemma}

\section{Large values over additive convolutions}
\label{sec:convolutions}
Given real numbers $N,\Delta$, a well spaced sequence $t_1,\dots,t_R$ and a positive real numbers $\gamma(t)$ we define
\begin{align}
\label{eq:SNDdef}
S(N,\Delta,\gamma)=\sum_{\substack{r,s=1 \\ |t_r-t_s|\le \Delta}}^{R}\gamma(t_r)\gamma(t_s)\left|\sum_{N\le n \le 2N}n^{i(t_r-t_s)}\right|^2,
\end{align}
\begin{align}
\label{eq:IDgdef}
I(\Delta,\gamma)=\sum_{\substack{r,s=1 \\ |t_r-t_s|\le \Delta}}^{R}\gamma(t_r)\gamma(t_s),\end{align}
and when $\gamma$ is the indicator function of some set $\cA$ let $I(\Delta,\cA)$ be as in~\eqref{eq:IdAset}. If each $t_i\le T$ then we simplify 
$$S(N,T,\gamma)=S(N,\gamma).$$
\begin{theorem}
\label{thm:largeadditive}
Let $T$ and $N$ be  real numbers, $1\le t_1,\dots,t_R\le T$ a well spaced sequence, $a_n$ a sequence of complex numbers satisfying
\begin{align*}
|a_n|\le 1,
\end{align*}
and $\gamma$ a sequence of positive real numbers. If $N\ge \Delta^{2/3}$ then we have 
\begin{align*}
S(N,\Delta,\gamma)\ll (I(\Delta,\gamma)+N\|\gamma\|_2^2)N^{o(1)},
\end{align*}
and in general
\begin{align*}
S(N,\Delta,\gamma)\ll (I(\Delta,\gamma)+N\|\gamma\|_2^2+\Delta^{1/2}I(\Delta,\gamma)^{1/4}\|\gamma\|_2^{3/2})N^{o(1)}.
\end{align*}
\end{theorem}
As a consequence of Theorem~\ref{thm:largeadditive} we have.
\begin{theorem}
\label{thm:largeadditive1}
Let $N,T\ge 1$ be  real numbers, $\cA\subseteq [0,T]$ a well spaced set and $a_n$ a sequence of complex numbers satisfying 
$$|a_n|\le 1.$$
We have 
\begin{align*}
& \sum_{\substack{t_i\in \cA }}\left|\sum_{N\le n\le 2N}a_nn^{-1/2+i(t_1+\dots-t_{2k})} \right|^2 \ll  \\ & \quad \quad \quad (|\cA|^{2k}+NT_k(\cA)+T^{1/2}|\cA|^{k/2}T_{k}(\cA)^{3/4})N^{o(1)}.
\end{align*}
In particular, if either 
\begin{align}
\label{eq:Ndeltacond}
N\ge T^{2/3} \quad \text{or} \quad T^{2/3}T_k(\cA)\le |\cA|^{2k},
\end{align}
then we have 
\begin{align*}
&\sum_{\substack{t_i\in \cA}}\left|\sum_{N\le n\le 2N}a_nn^{-1/2+i(t_1+\dots-t_{2k})} \right|^2\ll  \left(|\cA|^{2k}+NT_k(\cA)\right)(NT)^{o(1)}.
\end{align*}
\end{theorem}
The following preliminary results are required for the proof of Theorem~\ref{thm:largeadditive}.
\begin{lemma}
\label{lem:coefficients}
Let $t_r$ a sequence of real numbers, $a_n$ a sequence of complex numbers, $b_n$ a sequence of positive real numbers satisfying
\begin{align}
\label{eq:anUB}
|a_n|\le b_n,
\end{align}
and  $\gamma$ a sequence of positive real numbers.
For any positive  $\Delta,N,K$  and $\varepsilon$ satisfying
\begin{align}
\label{eq:coefficientsepsilon}
\varepsilon\le \frac{1}{4}, \quad K>1,
\end{align}
we have
\begin{align*}
&\sum_{\substack{r,s=1 \\ |t_r-t_s|\le \Delta }}^{R}\gamma(t_r)\gamma(t_s)\left|\sum_{N\le n \le KN}a_n n^{-1/2+i(t_r-t_s)}\right|^2\ll \\ & \quad \quad \quad \frac{1}{\varepsilon}\sum_{\substack{r,s=1 \\ |t_r-t_s|\le 4\varepsilon\Delta}}^{R}\gamma(t_r)\gamma(t_s)\left|\sum_{N\le n \le KN}b_n n^{-1/2+i(t_r-t_s)}\right|^2.
\end{align*}
\end{lemma}
\begin{proof}
Define
\begin{align}
\label{eq:Fdef}
F(x)=\max\{1-|x|,0\},
\end{align}
and note that 
\begin{align}
\label{eq:Fhat}
\widehat F(y)=\left(\frac{\sin{\pi y}}{\pi y}\right)^2\ge 0,
\end{align}
where $\widehat F$ denotes the Fourier transform. Since 
\begin{align*}
\widehat F\left(x\right)\gg1 \quad \text{if} \quad |x|\le \frac{1}{4},
\end{align*}
we have 
\begin{align}
\label{eq:Fhathat1}
& \sum_{\substack{r,s=1 \\ |t_r-t_s|\le \Delta }}^{R}\gamma(t_r)\gamma(t_s)\left|\sum_{N\le n \le KN}a_n n^{-1/2+i(t_r-t_s)}\right|^2\ll  \\ &  \sum_{\substack{r,s=1}}^{R}\gamma(t_r)\gamma(t_s)\widehat F\left(\frac{t_r-t_s}{4\Delta}\right)\left|\sum_{N\le n \le KN}a_n n^{-1/2+i(t_r-t_s)}\right|^2. \nonumber
\end{align}
Expanding the square, interchanging summation and using Fourier inversion, we get
\begin{align*}
&\sum_{\substack{r,s=1}}^{R}\gamma(t_r)\gamma(t_s)\widehat F\left(\frac{t_r-t_s}{4\Delta}\right)\left|\sum_{N\le n \le KN}a_n n^{-1/2+i(t_r-t_s)}\right|^2= \\ & \int_{-\infty}^{\infty}F(y)\sum_{N\le n_1,n_2\le KN}a_{n_1}\overline a_{n_2}(n_1n_2)^{-1/2} \\ & \quad \quad \quad \times \sum_{\substack{r,s=1}}^{R}\gamma(t_r)\gamma(t_s)e\left(y\left(\frac{t_r}{4\Delta}-\frac{t_s}{4\Delta} \right) \right)n_1^{i(t_r-t_s)}n_2^{-i(t_r-t_s)}dy,
\end{align*}
and hence by~\eqref{eq:anUB} 
\begin{align*}
&\sum_{\substack{r,s=1}}^{R} \gamma(t_r)\gamma(t_s)\widehat F\left(\frac{t_r-t_s}{4\Delta}\right)\left|\sum_{N\le n \le KN}a_n n^{-1/2+i(t_r-t_s)}\right|^2\le \\ & \int_{-\infty}^{\infty}F(y)\sum_{N\le n_1,n_2\le KN}b_{n_1}b_{n_2}(n_1n_2)^{-1/2}\left|\sum_{r=1}^{R}\gamma(t_r)e\left(\frac{yt_r}{4\Delta}\right)n_1^{it_r}n_2^{-it_r} \right|^2dy.
\end{align*}
Since 
\begin{align}
\label{eq:Fycondcond}
F(y)=0 \quad \text{if} \quad |y|\ge 1 \quad \text{and} \quad  F(y)\le 1 \quad \text{otherwise},
\end{align}
and 
\begin{align*}
\widehat F(\varepsilon y)\gg 1 \quad \text{if} \quad y\le \frac{1}{4\varepsilon},
\end{align*}
the assumption~\eqref{eq:coefficientsepsilon} implies that
\begin{align*}
&\sum_{\substack{r,s=1}}^{R}\gamma(t_r)\gamma(t_s)\widehat F\left(\frac{t_r-t_s}{4\Delta}\right)\left|\sum_{N\le n \le KN}a_n n^{-1/2+i(t_r-t_s)}\right|^2\ll \\ & \int_{-\infty}^{\infty}\widehat F(\varepsilon y)\sum_{N\le n_1,n_2\le 2N}b_{n_1}b_{n_2}(n_1n_2)^{-1/2}\left|\sum_{r=1}^{R}\gamma(t_r)e\left(\frac{yt_r}{4\Delta}\right)n_1^{it_r}n_2^{-it_r} \right|^2dy,
\end{align*}
and hence 
\begin{align*}
&\sum_{\substack{r,s=1}}^{R}\gamma(t_r)\gamma(t_s)\widehat F\left(\frac{t_r-t_s}{4\Delta}\right)\left|\sum_{N\le n \le 2N}a_n n^{-1/2+i(t_r-t_s)}\right|^2\ll \\ & \frac{1}{\varepsilon}\int_{-\infty}^{\infty}\widehat F(y)\sum_{N\le n_1,n_2\le 2N}b_{n_1}b_{n_2}(n_1n_2)^{-1/2}\left|\sum_{r=1}^{R}\gamma(t_r)e\left(\frac{yt_r}{4\varepsilon \Delta}\right)n_1^{it_r}n_2^{-it_r} \right|^2dy.
\end{align*}

Expanding the square, interchanging summation and using Fourier inversion, we get 
\begin{align*}
&\int_{-\infty}^{\infty}\widehat F(y)\sum_{N\le n_1,n_2\le 2N}b_{n_1}b_{n_2}(n_1n_2)^{-1/2}\left|\sum_{r=1}^{R}\gamma(t_r)e\left(\frac{yt_r}{4\varepsilon \Delta}\right)n_1^{it_r}n_2^{-it_r} \right|^2dy= \\ &
\quad \quad \quad \sum_{\substack{r,s=1}}^{R}\gamma(t_r)\gamma(t_s)\int_{-\infty}^{\infty}\widehat F(y)e\left(y\left(\frac{t_r}{4\varepsilon\Delta}-\frac{t_s}{4\varepsilon\Delta} \right) \right)dy \\ & \quad \quad \quad \quad \times \sum_{N\le n_1,n_2\le 2N}b_{n_1}b_{n_2}n_1^{-1/2+i(t_r-t_s)}n_2^{-1/2-i(t_r-t_s)},
\end{align*}
which after rearranging gives
\begin{align*}
&\sum_{\substack{r,s=1}}^{R}\gamma(t_r)\gamma(t_s)\widehat F\left(\frac{t_r-t_s}{4\Delta}\right)\left|\sum_{N\le n \le 2N}a_n n^{-1/2+i(t_r-t_s)}\right|^2 \\ &  \quad \quad \quad \quad \ll \frac{1}{\varepsilon}\sum_{r,s=1}^{R}\gamma(t_r)\gamma(t_s)F\left(\frac{t_r-t_s}{4\varepsilon \Delta}\right)\left|\sum_{N\le n \le 2N}b_n n^{-1/2+i(t_r-t_s)} \right|^2,
\end{align*}
and the result follows from~\eqref{eq:Fhathat1} and~\eqref{eq:Fycondcond}.
\end{proof}
\begin{lemma}
\label{lem:smoothsums}
Let $N$ be an integer, $t_r$ a sequence of real numbers, $a_n$ a sequence of complex numbers satisfying
\begin{align*}
|a_n|\le 1,
\end{align*}
and $\gamma$ a sequence of positive real numbers. Let $c_1<c_2$ be constants and for each pair of integers $r,s$ let $N_{r,s}$ and $M_{r,s}$ be integers satisfying
\begin{align*}
c_1N\le N_{r,s}< M_{r,s}\le c_2N.
\end{align*}
For any $\Delta \ge 1$ we have 
\begin{align*}
& \sum_{\substack{r,s=1 \\ |t_r-t_s|\le \Delta}}^{R}\gamma(t_r)\gamma(t_s)\left|\sum_{N_{r,s}\le n \le M_{r,s}}a_n n^{i(t_r-t_s)}\right|^2 \\ &  \quad \quad \quad \quad \quad \ll (\log{N})^2\sum_{\substack{r,s=1 \\ |t_r-t_s|\le \Delta}}^{R}\gamma(t_r)\gamma(t_s)\left|\sum_{c_1N\le n \le c_2N}n^{i(t_r-t_s)}\right|^2.
\end{align*}
\begin{proof}
Let 
\begin{align*}
S=\sum_{\substack{r,s=1 \\ |t_r-t_s|\le \Delta}}^{R}\gamma(t_r)\gamma(t_s)\left|\sum_{N_{r,s}\le n \le M_{r,s}}a_n n^{i(t_r-t_s)}\right|^2.
\end{align*}
For each pair $r,s$ we have 
\begin{align*}
&\sum_{N_{r,s}\le n \le M_{r,s}}a_n n^{i(t_r-t_s)}= \\ & \int_{0}^{1}\left(\sum_{N_{r,s}\le n \le M_{r,s}}e(\alpha n)\right)\sum_{c_1N\le n \le c_2N}a_ne(\alpha n) n^{i(t_r-t_s)}d\alpha,
\end{align*}
so that defining
\begin{align*}
F(\alpha)=\min\left(N,\frac{1}{\|\alpha\|}\right),
\end{align*}
we have 
\begin{align*}
S\ll \sum_{\substack{r,s=1 \\ |t_r-t_s|\le \Delta}}^{R}\gamma(t_r)\gamma(t_s)\left(\int_{0}^{1}F(\alpha)\left|\sum_{c_1N\le n \le c_2N}a_ne(\alpha n) n^{i(t_r-t_s)} \right|d\alpha \right)^2.
\end{align*}
By the Cauchy-Schwarz inequality 
\begin{align*}
S\ll \left(\int_{0}^{1}F(\alpha)d\alpha\right)\sum_{\substack{r,s=1 \\ |t_r-t_s|\le \Delta}}^{R}\gamma(t_r)\gamma(t_s)\int_{0}^{1}F(\alpha)\left|\sum_{c_1N\le n \le c_2N}a_ne(\alpha n) n^{i(t_r-t_s)} \right|^2d\alpha.
\end{align*}
Since 
\begin{align*}
\int_{0}^{1}F(\alpha)d\alpha\ll  \log{N},
\end{align*}
we see that 
\begin{align*}
S\ll \log{N}\int_{0}^{1}F(\alpha)\sum_{\substack{r,s=1 \\ |t_r-t_s|\le \Delta}}^{R}\gamma(t_r)\gamma(t_s)\left|\sum_{c_1N\le n \le c_2N}a_ne(\alpha n) n^{i(t_r-t_s)} \right|^2d\alpha.
\end{align*}
Applying Lemma~\ref{lem:coefficients} with $\varepsilon=1/4$ gives
\begin{align*}
S&\ll \log{N}\int_{0}^{1}F(\alpha)d\alpha \sum_{\substack{r,s=1 \\ |t_r-t_s|\le \Delta}}^{R}\gamma(t_r)\gamma(t_s)\left|\sum_{c_1N\le n \le c_2N}n^{i(t_r-t_s)} \right|^2 \\
&\ll (\log{N})^2\sum_{\substack{r,s=1 \\ |t_r-t_s|\le \Delta}}^{R}\gamma(t_r)\gamma(t_s)\left|\sum_{c_1N\le n \le c_2N}n^{i(t_r-t_s)} \right|^2,
\end{align*}
and completes the proof.
\end{proof}
\end{lemma}
Our next two results are variants of~\cite[Lemma~2]{HB}.
\begin{lemma}
\label{lem:larger}
Let  $N,K\ge 1$ be real numbers. Suppose $t_r$ is a sequence of real numbers satisfying
\begin{align*}
1\le t_r \le T.
\end{align*}
 For any  $U\ge 1$  we have
we have 
\begin{align*}
&\sum_{\substack{r,s=1 \\ |t_r-t_s|\le \Delta}}^{R}\gamma(t_r)\gamma(t_s)\left|\sum_{N\le n \le KN}n^{-1/2+i(t_r-t_s)}\right|^2 \\ & \ll  (KNU)^{o(1)}\sum_{\substack{r,s=1 \\ |t_r-t_s|\le \Delta}}^{R}\gamma(t_r)\gamma(t_s)\left|\sum_{UN\le n \le 3KNU/2}n^{-1/2+i(t_r-t_s)}\right|^2.
\end{align*}
\end{lemma}
\begin{proof}
Let $P$ be a prime number satisfying
\begin{align}
\label{eq:PUcond}
8U\le P \le 16U,
\end{align}
and for each multiplicative character $\chi$ mod $P$ let
\begin{align*}
S(\chi)=\sum_{\substack{r,s=1 \\ |t_r-t_s|\le \Delta}}^{R}\gamma(t_r)\gamma(t_s)\left|\sum_{N\le n \le KN}n^{-1/2+i(t_r-t_s)}\right|^2\left|\sum_{U\le u \le 3U/2}\chi(u)u^{-1/2+i(t_r-t_s)} \right|^2.
\end{align*}
For integer $m$ define
\begin{align*}
a(m,\chi)=\sum_{\substack{U\le u \le 3U/2 \\ N\le n \le KN \\ un=m}}\chi(u),
\end{align*}
so that 
\begin{align*}
a(m,\chi)\ll (KNU)^{o(1)},
\end{align*}
and hence by Lemma~\ref{lem:coefficients}
\begin{align*}
S(\chi)\ll (KNU)^{o(1)}\sum_{\substack{r,s=1 \\ |t_r-t_s|\le \Delta}}^{R}\gamma(t_r)\gamma(t_s)\left|\sum_{NU\le n \le 3KNU/2}n^{-1/2+i(t_r-t_s)}\right|^2,
\end{align*}
which implies 
\begin{align}
\label{eq:SchiUB}
\sum_{\chi \mod{P}}S(\chi)\ll P(KNU)^{o(1)}\sum_{\substack{r,s=1 \\ |t_r-t_s|\le \Delta}}^{R}\gamma(t_r)\gamma(t_s)\left|\sum_{NU\le n \le 3KNU/2}n^{-1/2+i(t_r-t_s)}\right|^2.
\end{align}
We have 
\begin{align*}
&\sum_{\chi \mod{P}}S(\chi)=\sum_{\substack{r,s=1 \\ |t_r-t_s|\le \Delta}}^{R}\gamma(t_r)\gamma(t_s)\left|\sum_{N\le n \le KN}n^{-1/2+i(t_r-t_s)}\right|^2 \\ & \times\sum_{U\le u_1,u_2\le 3U/2}u_1^{-1/2+i(t_r-t_s)}u_2^{-1/2-i(t_r-t_s)}\sum_{\chi \mod{P}}\chi(u_1)\overline \chi(u_2),
\end{align*}
hence by~\eqref{eq:PUcond} and orthogonality of characters
\begin{align*}
\sum_{\chi \mod{P}}S(\chi)\gg P\log{U}\sum_{\substack{r,s=1 \\ |t_r-t_s|\le \Delta}}^{R}\gamma(t_r)\gamma(t_s)\left|\sum_{N\le n \le KN}n^{-1/2+i(t_r-t_s)}\right|^2,
\end{align*}
and the result follows from~\eqref{eq:SchiUB}.
\end{proof}
\begin{cor}
\label{cor:larger}
 For any  integer $M\ge 2N$  we have
we have 
\begin{align*}
S(N,\Delta,\gamma)\ll  M^{o(1)}S(M,\Delta,\gamma).
\end{align*}
\end{cor}
\begin{proof}
Writing $K=2^{1/3}$ we have 
\begin{align}
\label{eq:kkkKKkKkKkK}
S(N,\Delta,\gamma)\ll \sum_{i=0}^{2}S_i,
\end{align}
where 
\begin{align*}
S_i=\sum_{\substack{r,s=1 \\ |t_r-t_s|\le \Delta}}^{R}\gamma(t_r)\gamma(t_s)\left|\sum_{K^{i}N\le n \le K^{i+1}N}n^{-1/2+i(t_r-t_s)}\right|^2.
\end{align*}
Define
\begin{align*}
M_i=\frac{M}{K^{i}N},
\end{align*}
and apply Lemma~\ref{lem:larger} to get 
\begin{align*}
S_i\ll M^{o(1)}\sum_{\substack{r,s=1 \\ |t_r-t_s|\le \Delta}}^{R}\gamma(t_r)\gamma(t_s)\left|\sum_{M\le n \le 3KM/2}n^{-1/2+i(t_r-t_s)}\right|^2.
\end{align*}
By Lemma~\ref{lem:coefficients}
\begin{align*}
S_i\ll  M^{o(1)}\sum_{\substack{r,s=1 \\ |t_r-t_s|\le \Delta}}^{R}\gamma(t_r)\gamma(t_s)\left|\sum_{M\le n \le 2M}n^{-1/2+i(t_r-t_s)}\right|^2,
\end{align*}
and the result follows from~\eqref{eq:kkkKKkKkKkK}.
\end{proof}
The following is our variant of~\cite[Lemma~4]{HB}.
\begin{lemma}
\label{lem:square}
For any $M\ge 8N^2$ we have 
\begin{align*}
S(N,\Delta,\gamma)^2\ll M^{o(1)}I(\Delta,\gamma)S(M,\Delta,\gamma).
\end{align*}
\end{lemma}
\begin{proof}
By the Cauchy-Schwarz inequality and Lemma~\ref{lem:coefficients}
\begin{align*}
S(N,\Delta,\gamma)^2&\le I(\Delta,\gamma)\sum_{\substack{r,s=1 \\ |t_r-t_s|\le \Delta}}^{R}\gamma(t_r)\gamma(t_s)\left|\sum_{N\le n \le 2N}n^{-1/2+i(t_r-t_s)}\right|^4 \\
&\ll N^{o(1)}I(\Delta,\gamma)\sum_{\substack{r,s=1 \\ |t_r-t_s|\le \Delta}}^{R}\gamma(t_r)\gamma(t_s)\left|\sum_{N^2\le n \le 4N^2}n^{-1/2+i(t_r-t_s)}\right|^2.
\end{align*}
We have 
\begin{align*}
&\sum_{\substack{r,s=1 \\ |t_r-t_s|\le \Delta}}^{R}\gamma(t_r)\gamma(t_s)\left|\sum_{N^2\le n \le 4N^2}n^{-1/2+i(t_r-t_s)}\right|^2\ll \\ & \sum_{\substack{r,s=1 \\ |t_r-t_s|\le \Delta}}^{R}\gamma(t_r)\gamma(t_s)\left|\sum_{N^2\le n \le 2N^2}n^{-1/2+i(t_r-t_s)}\right|^2 \\ &+\sum_{\substack{r,s=1 \\ |t_r-t_s|\le \Delta}}^{R}\gamma(t_r)\gamma(t_s)\left|\sum_{2N^2\le n \le 4N^2}n^{-1/2+i(t_r-t_s)}\right|^2,
\end{align*}
and the result follows from the above and Corollary~\ref{cor:larger}.
\end{proof}
The following is a variant of Hilbert's inequality. 
\begin{lemma}
\label{lem:hilbert}
For any well spaced sequence $t_1,\dots,t_R$ and positive real numbers $\gamma(t_r)$  we have 
\begin{align*}
\sum_{\substack{r\neq s}}\frac{\gamma(t_r)\gamma(t_s)}{(t_r-t_s)^2}\ll \|\gamma_2\|^2.
\end{align*}
\end{lemma}
\begin{proof}
Let 
\begin{align*}
S=\sum_{\substack{r\neq s}}\frac{\gamma(t_r)\gamma(t_s)}{(t_r-t_s)^2}.
\end{align*}
By the Cauchy-Schwarz inequality 
\begin{align*}
S^2&=\left(\sum_{\substack{r\neq s}}\frac{\gamma(t_r)}{(t_r-t_s)}\frac{\gamma(t_s)}{(t_r-t_s)}\right)^2\le \left(\sum_{\substack{r\neq s}}\frac{\gamma(t_r)^2}{(t_r-t_s)^2} \right)^2.
\end{align*}
We have 
\begin{align*}
\sum_{\substack{r\neq s}}\frac{\gamma(t_r)^2}{(t_r-t_s)^2}= \sum_{r}\gamma(t_r)^2\sum_{s\neq r}\frac{1}{(t_r-t_s)^2},
\end{align*}
and the result follows since the assumption $t_r$ is well spaced implies for any fixed $r$
\begin{align*}
\sum_{s\neq r}\frac{1}{(t_r-t_s)^2}\ll 1.
\end{align*}
\end{proof}
\begin{lemma}
\label{lem:mvSmall}
Let $N\ge 2$ be an integer, $\Delta \gg1$ a real number, $1\le t_1,\dots,t_{R}\le T$ a well spaced sequence 
and $\gamma(t)$ a sequence of positive real numbers. We have 
\begin{align*}
\sum_{\substack{r,s=1 \\ |t_r-t_s|\le \Delta }}^{R}\gamma(t_r)\gamma(t_s)\left|\sum_{N\le n \le 2N} n^{-1/2+i(t_r-t_s)} \right|^2\ll N\|\gamma\|_2^{2}+\frac{\Delta^{1+o(1)}}{N}I(\Delta,\gamma).
\end{align*}
\end{lemma}
\begin{proof}
Let 
\begin{align*}
S=\sum_{\substack{r,s=1 \\ |t_r-t_s|\le \Delta }}^{R}\gamma(t_r)\gamma(t_s)\left|\sum_{N\le n \le 2N} n^{-1/2+i(t_r-t_s)} \right|^2.
\end{align*}
By Lemma~\ref{lem:coefficients}, we have 
\begin{align}
\label{eq:SS12small}
S&\ll \frac{1}{N}\sum_{\substack{r,s=1 \\ |t_r-t_s|\le \Delta }}^{R}\gamma(t_r)\gamma(t_s)\left|\sum_{N\le n \le 2N} n^{i(t_r-t_s)} \right|^2 \\ 
&\ll \frac{1}{N}\left( S_1+S_2\right), \nonumber
\end{align}
where 
\begin{align*}
S_1=\sum_{\substack{r,s=1 \\ |t_r-t_s|< 1 }}^{R}\gamma(t_r)\gamma(t_s)\left|\sum_{1\le n \le N}n^{i(t_r-t_s)} \right|^2,
\end{align*}
and 
\begin{align*}
S_2=\sum_{\substack{r,s=1 \\ 1\le |t_r-t_s|\le \Delta }}^{R}\gamma(t_r)\gamma(t_s)\left|\sum_{1\le n \le N}n^{i(t_r-t_s)} \right|^2.
\end{align*}
We estimate the summation in $S_1$ trivially. Using the assumption the points $t_r$ are well spaced, we have 
\begin{align}
\label{eq:S1b1b1b1}
S_1\ll N\sum_{\substack{r,s=1 \\ |t_r-t_s|< 1 }}^{R}\gamma(t_r)\gamma(t_s)\ll N\|\gamma\|_2^2.
\end{align}
For $S_2$ we use the estimate
\begin{align*}
\sum_{N\le n \le 2N}n^{it}\ll \frac{N}{t}+t^{1/2}(\log{t}),
\end{align*}
valid for any $N,t\ge1$, see for example~\cite[Equation~9.21]{IwKo}. This gives 
\begin{align*}
S_2\ll N^2\sum_{\substack{r,s=1 \\ 1\le |t_r-t_s|\le \Delta }}^{R}\frac{\gamma(t_r)\gamma(t_s)}{(t_r-t_s)^2}+\Delta^{1+o(1)}\sum_{\substack{r,s=1 \\ 1\le |t_r-t_s|\le \Delta }}^{R}\gamma(t_r)\gamma(t_s),
\end{align*}
which by Lemma~\ref{lem:hilbert} implies
$$S_2\ll N^2\|\gamma\|_2^2+\Delta^{1+o(1)}I(\Delta,\gamma).$$
Combining the above with~\eqref{eq:SS12small} and~\eqref{eq:S1b1b1b1}
\begin{align*}
S\ll N\|\gamma\|_2^2+\frac{\Delta^{1+o(1)}}{N}I(\Delta,\gamma),
\end{align*}
and completes the proof.
\end{proof}
The following forms the basis of the van der Corput method of exponential sums, for a proof  see~\cite[Theorem~8.16]{IwKo}.
\begin{lemma}
\label{lem:vdc}
For any real valued function $f$ defined on an interval $[a,b]$ with derivatives satisfying
\begin{align*}
\frac{T}{N^2} \ll f''(z) \ll \frac{T}{N^2}, \quad |f^{(3)}(z)|\ll \frac{T}{N^3}, \quad |f^{(4)}(z)|\ll \frac{T}{N^4},\quad z\in[a,b],
\end{align*}
we have
\begin{align*}
\sum_{a<n<b}e(f(n))=\sum_{\alpha<m<\beta}f''(x_m)^{-1/2}e\left(f(x_m)-mx_m+\frac{1}{8} \right)+E,
\end{align*}
where $\alpha=f'(a), \beta=f'(b)$, $x_m$ is the unique solution to $f'(x)=m$ for $x\in [a,b]$ and
\begin{align*}
E&\ll \frac{N}{T^{1/2}}+\log(|f'(b)-f'(a)|+2).
\end{align*}
\end{lemma}
The following is a consequence of Lemma~\ref{lem:vdc} and partial summation.
\begin{lemma}
\label{lem:vdc1}
Let $N$ be an integer and $t\ge 1$ a real number. We have 
\begin{align*}
\left|\sum_{N\le n \le 2N}n^{it}\right|\ll \frac{N}{t^{1/2}}\max_{M\le t/N}\left|\sum_{t/2N\le n \le M}n^{it} \right|+\frac{N}{t^{1/2}}+\log{(Nt)}.
\end{align*}
\end{lemma}

\begin{lemma}
\label{lem:main1}
Let $T$ be a real number, $1\le t_1,\dots,t_R\le T$ a well spaced sequence. For integer $N$ and a real number $\Delta$ we define
\begin{align*}
S^*(N,\Delta,\gamma)=\sum_{\substack{r,s=1 \\ \Delta\le |t_r-t_s|\le 2\Delta}}^{R}\gamma(t_r)\gamma(t_s)\left|\sum_{N\le n \le 2N}n^{-1/2+i(t_r-t_s)}\right|^2,
\end{align*}
If $\Delta \gg N$ then
\begin{align*}
S^{*}(N,\Delta)&\ll
N^{o(1)}\sum_{\substack{r,s=1 \\ |t_r-t_s|\le 2\Delta }}^{R}\gamma(t_r)\gamma(t_s)\left|\sum_{\Delta/2N\le n \le 2\Delta/N}n^{-1/2+i(t_r-t_s)} \right|^2 \\ & \quad \quad +N^{o(1)}I(\Delta,\gamma).
\end{align*}
\end{lemma}
\begin{proof}
Let 
\begin{align*}
S_0^*(N,\Delta,\gamma)=\sum_{\substack{r,s=1 \\ \Delta\le |t_r-t_s|\le 2\Delta}}^{R}\gamma(t_r)\gamma(t_s)\left|\sum_{N\le n \le 2N}n^{i(t_r-t_s)}\right|^2,
\end{align*}
so that by Lemma~\ref{lem:coefficients}
\begin{align}
\label{eq:SS0starttttt}
S^*(N,\Delta,\gamma)\ll \frac{S_0^*(N,\Delta,\gamma)}{N}\ll S^*(N,\Delta,\gamma).
\end{align}
 For each pair $r,s$ satisfying $$\Delta \le |t_r-t_s|\le 2\Delta,$$ we define  $\Delta_{r,s}=|t_r-t_s|$. By Lemma~\ref{lem:vdc1} and the assumption $\Delta \gg N$
\begin{align*}
\left|\sum_{N\le n \le 2N}n^{i(t_r-t_s)}\right|\ll \frac{N}{\Delta^{1/2}}\max_{M\le \Delta/N}\left|\sum_{\Delta_{r,s}/2N\le n \le M}n^{i(t_r-t_s)} \right|+N^{1/2+o(1)},
\end{align*}
which after summing over $r,s$ gives
\begin{align*}
S_0^*(N,\Delta,\gamma)
&\ll \frac{N^2}{\Delta}\sum_{\substack{r,s=1 \\ \Delta\le |t_r-t_s|< 2\Delta}}^{R}\gamma(t_r)\gamma(t_s)\max_{M\le \Delta/N}\left|\sum_{\Delta_{r,s}/2N\le n \le M}n^{i(t_r-t_s)} \right|^2  \\ & \quad \quad \quad \quad +I(\Delta,\gamma)N^{1+o(1)}.
\end{align*}
We have
\begin{align*}
&\sum_{\substack{r,s=1 \\ \Delta\le |t_r-t_s|< 2\Delta}}^{R}\gamma(t_r)\gamma(t_s)\max_{M\le \Delta/N}\left|\sum_{\Delta_{r,s}/2N\le n \le M}n^{i(t_r-t_s)} \right|^2\le \\ & \quad \quad \quad  \sum_{\substack{\substack{r,s=1 \\ |t_r-t_s|\le 2 \Delta}}}^{R}\gamma(t_r)\gamma(t_s)\max_{M\le \Delta/N}\left|\sum_{\Delta_{r,s}/2N\le n \le M}n^{i(t_r-t_s)} \right|^2,
\end{align*}
and hence by Lemma~\ref{lem:smoothsums} 
\begin{align*}
S_0^{*}(N,\Delta,\gamma)&\ll
\frac{N^{2+o(1)}}{\Delta}\sum_{\substack{r,s=1 \\ |t_r-t_s|\le 2\Delta }}^{R}\gamma(t_r)\gamma(t_s)\left|\sum_{\Delta/2N\le n \le 2\Delta/N}n^{i(t_r-t_s)} \right|^2 \\ & \quad \quad \quad+I(\Delta,\gamma)N^{1+o(1)}.
\end{align*}
Using~\eqref{eq:SS0starttttt} and a second application of Lemma~\ref{lem:coefficients} we get 
\begin{align*}
S^{*}(N,\Delta,\gamma)&\ll
N^{o(1)}\sum_{\substack{r,s=1 \\ |t_r-t_s|\le 2\Delta }}^{R}\gamma(t_r)\gamma(t_s)\left|\sum_{\Delta/2N\le n \le 2\Delta/N}n^{-1/2+i(t_r-t_s)} \right|^2 \\ & \quad \quad \quad+I(\Delta,\gamma)N^{o(1)},
\end{align*}
and the result follows after splitting summation over $n$ into dyadic ranges and applying Corollary~\ref{cor:larger}.
\end{proof}

\begin{cor}
\label{cor:reflection}
Let $T$ be a real number, $1\le t_1,\dots,t_R\le T$ a well spaced sequence and $\gamma(t)$ a sequence of positive real numbers. We have 
\begin{align*}
S(N,\Delta,\gamma)\ll N^{o(1)}S\left(\frac{4\Delta}{N},\Delta,\gamma\right)+N^{o(1)}I(\Delta,\gamma)+\|\gamma\|_2N^{1+o(1)}.
\end{align*}
\end{cor}
\begin{proof}
We may suppose $\Delta \gg N$ since otherwise the result follows from Lemma~\ref{lem:mvSmall}. With notation as in Lemma~\ref{lem:main1}, we have 
\begin{align*}
\nonumber S(N,\Delta,\gamma)&\le \sum_{\substack{r,s=1 \\  |t_r-t_s|\le 10N}}^{R}\gamma(t_r)\gamma(t_s)\left|\sum_{N\le n \le 2N}n^{-1/2+i(t_r-t_s)}\right|^2 \\ &  \quad \quad \quad +\sum_{\substack{r,s=1 \\  |t_r-t_s|>  10N}}^{R}\gamma(t_r)\gamma(t_s)\left|\sum_{N\le n \le 2N}n^{-1/2+i(t_r-t_s)}\right|^2,
\end{align*}
and hence 
\begin{align*}
S(N,\Delta,\gamma)\ll\sum_{\substack{r,s=1 \\  |t_r-t_s|\le N}}^{R}\gamma(t_r)\gamma(t_s)\left|\sum_{N\le n \le 2N}n^{-1/2+i(t_r-t_s)}\right|^2+(\log{\Delta})S^{*}(N,\Delta_1,\gamma),
\end{align*}
for some 
\begin{align*}
10N\le \Delta_1\le 2\Delta.
\end{align*}
By Lemma~\ref{lem:mvSmall}
\begin{align}
\label{eq:cormain2}
\nonumber \sum_{\substack{r,s=1 \\  |t_r-t_s|\le N}}^{R}\gamma(t_r)\gamma(t_s)\left|\sum_{N\le n \le 2N}n^{-1/2+i(t_r-t_s)}\right|^2&\ll \left(\|\gamma\|_2^2N+I(10N,\gamma)\right)N^{o(1)} \\ 
&\ll \left(\|\gamma\|_2^2N+I(\Delta,\gamma)\right)N^{o(1)},
\end{align}
and by Corollary~\ref{cor:larger} and Lemma~\ref{lem:main1}
\begin{align*}
S^{*}(N,\Delta_1)& \ll  N^{o(1)}S\left(\frac{2\Delta_1}{N},\Delta_1 \right)+I(\Delta_1,\gamma) \\ &\ll N^{o(1)}S\left(\frac{4\Delta}{N},\Delta \right)+I(2\Delta,\gamma),
\end{align*}
from which the desired result follows after combining with Lemma~\ref{lem:coefficients} and~\eqref{eq:cormain2}.
\end{proof}

\section{Proof of Theorem~\ref{thm:largeadditive}}
First consider when
\begin{align}
\label{eq:laNub1}
N\ge \frac{\Delta^{2/3}}{100}.
\end{align}
 By Corollary~\ref{cor:reflection}
\begin{align*}
S(N,\Delta,\gamma)\ll N^{o(1)}S\left(\frac{4\Delta}{N},\Delta,\gamma \right)+(I(\Delta,\gamma)+N\|\gamma\|_2^2)N^{o(1)},
\end{align*}
and by Lemma~\ref{lem:square}
\begin{align}
\label{eq:SNr}
S(N,\Delta,\gamma)\ll N^{o(1)}I(\Delta,\gamma)^{1/2}S\left(\frac{100\Delta^2}{N^2},\Delta,\gamma \right)^{1/2}+(I(\Delta,\gamma)+\|\gamma\|_2^2)N^{o(1)}.
\end{align}
By~\eqref{eq:laNub1} 
\begin{align}
\frac{100\Delta^2}{N^2}\le \frac{N}{2},
\end{align}
and hence by Corollary~\ref{cor:larger}
\begin{align}
\label{eq:SNcase1}
S(N,\Delta,\gamma)\ll N^{o(1)}I(\Delta,\gamma)^{1/2}S\left(N,\Delta,\gamma \right)^{1/2}+(I(\Delta,\gamma)+N\|\gamma\|_2^2)N^{o(1)},
\end{align}
which implies 
\begin{align}
\label{eq:SNcase11}
S(N,\Delta,\gamma)\ll (I(\Delta,\gamma)+N\|\gamma\|_2^2)N^{o(1)}. 
\end{align}
If~\eqref{eq:laNub1} does not hold then 
\begin{align*}
\frac{\Delta^2}{N^2}\gg \Delta^{2/3},
\end{align*}
and hence by Corollary~\ref{cor:larger} and~\eqref{eq:SNcase11}
\begin{align*}
S\left(\frac{100\Delta^2}{N^2},\Delta,\gamma \right)\ll \left(I(\Delta,\gamma)+\frac{\Delta^2}{N^2}\|\gamma\|_2^2\right)N^{o(1)}.
\end{align*}
Substituting into~\eqref{eq:SNcase1} gives 
\begin{align}
\label{eq:SNprelimF}
S(N,\Delta,\gamma)\ll \left(I(\Delta,\gamma)+\frac{\Delta}{N}I(\Delta,\gamma)^{1/2}\|\gamma\|_2+N\|\gamma\|_2^2\right)N^{o(1)}.
\end{align}
Let 
\begin{align}
\label{eq:Udef154}
U=\max \left\{2,\left(\frac{\Delta I(\Delta,\gamma)^{1/2}}{N^2\|\gamma\|_2}\right)^{1/2}  \right\}.
\end{align}
By Corollary~\ref{cor:larger} and~\eqref{eq:SNprelimF}
\begin{align}
\label{eq:SNprelimF1}
S(N,\Delta,\gamma) & \ll N^{o(1)}S(UN,\Delta,\gamma) \nonumber \\ 
&\ll (I(\Delta,\gamma)+\frac{\Delta}{NU}I(\Delta,\gamma)^{1/2}\|\gamma\|_2+NU\|\gamma\|_2^2)N^{o(1)}.
\end{align}
Using either~\eqref{eq:SNprelimF} or~\eqref{eq:SNprelimF1} depending of which term in~\eqref{eq:Udef154} is maximum
gives the bound 
\begin{align*}
S(N,\gamma)\ll (I(\Delta,\gamma)+N\|\gamma\|_2^2+\Delta^{1/2}I(\Delta,\gamma)^{1/4}\|\gamma\|_2^{3/2})N^{o(1)},
\end{align*}
which completes the proof.
\section{Proof of Theorem~\ref{thm:largeadditive1}}
Let 
$$S=\sum_{\substack{t_i\in \cA }}\left|\sum_{N\le n\le 2N}a_nn^{-1/2+i(t_1+\dots - t_{2k})} \right|^2.$$
For each integer $|\ell|\ll T$ we define
\begin{align*}
\gamma(\ell)=|\{ (t_1,\dots,t_k)\in \cA^k \ : \  \ell<t_1+\dots +t_k\le \ell+1\}|,
\end{align*}
and note that 
\begin{align}
\label{eq:gammaell1}
|\cA|^{2k}\ll I(T,\gamma)=\|\gamma\|^2_1\ll |\cA|^{2k},
\end{align}
and 
\begin{align}
\label{eq:gammaell2}
T_{2k}(\cA)\ll \|\gamma\|_2^2\ll T_{2k}(\cA).
\end{align}
We have
\begin{align*}
&S\le \sum_{\substack{|\ell_1|,|\ell_2|\le T \\ |\ell_1-\ell_2|\le 2\Delta}}\gamma(\ell_1)\gamma(\ell_2)\max_{|\theta|\le 2}\left|\sum_{N\le n\le 2N}a_nn^{-1/2+i(\ell_1-\ell_2)+i\theta} \right|^2,
\end{align*}
which combined with Lemma~\ref{lem:removemax} gives
\begin{align*}
S& \ll N^{o(1)}\int_{|\tau|\le 2\log{N}}\sum_{\substack{|\ell_1|,|\ell_2|\le T \\ |\ell_1-\ell_2|\le 2\Delta}}\gamma(\ell_1)\gamma(\ell_2)\left|\sum_{N\le n\le 2N}a_nn^{-1/2+i(\ell_1-\ell_2)+i\tau} \right|^2d\tau.
\end{align*}
Taking a maximum over $\tau$ this implies that 
\begin{align}
\label{eq:Sle1}
S\ll N^{o(1)}\sum_{\substack{|\ell_1|,|\ell_2|\le T \\ |\ell_1-\ell_2|\le 2\Delta}}\gamma(\ell_1)\gamma(\ell_2)\left|\sum_{N\le n\le 2N}a_nn^{i\tau_0}n^{-1/2+i(\ell_1-\ell_2)} \right|^2,
\end{align}
for some $\tau_0\le 2\log{N}.$ If $N\ge T^{2/3}$ then by Theorem~\ref{thm:largeadditive} 
\begin{align*}
S\ll (|\cA|^{2k}+NT_k(\cA)+T^{1/2}|\cA|^{k/2}T_{k}(\cA)^{3/4})N^{o(1)}.
\end{align*}
If  $N\le T^{2/3}$ then we have 
\begin{align*}
S\ll (|\cA|^{2k}+NT_k(\cA)+T^{1/2}|\cA|^{k/2}T_{k}(\cA)^{3/4})N^{o(1)},
\end{align*}
and the result follows noting that the conditions~\eqref{eq:Ndeltacond}  imply
\begin{align*}
T^{1/2}|\cA|^{k/2}T_{k}(\cA)^{3/4}\ll |\cA|^{2k}.
\end{align*}

\section{Large values of Dirichlet Polynomials}
In this section we state some reductions from large values of Dirichlet polynomials to various mean values which are slight modifications of well known results.  We first recall ~\cite[Lemma~1]{Jut}.
\begin{lemma}
\label{lem:jut}
Suppose  $T$ and $N$ satisfy $N\le T$. Let $h=(\log{T})^2$ and define
$$b(n)=e^{-(n/2N)^{h}}-e^{-(n/N)^h}.$$
Then for  any $t$ and $M$ satisfying
$$\frac{t}{N}\ll M\le T^2, \quad h^2\le t \le T,$$
we have
\begin{align*}
\sum_{n=1}^{\infty}b(n)n^{-it}\ll N^{1/2}\int_{|\tau|\le h^2}\left|\sum_{n=1}^{M}n^{-1/2+i(t+\tau)}\right|d\tau+1.
\end{align*}
\end{lemma}
The following is a consequence of Mellin inversion, see for example~\cite[Equation~4.17]{Bou}.
\begin{lemma}
\label{lem:jut1}
Let $N,T\gg 1$ and define 
\begin{align*}
c_n=e^{-n/2N}-e^{-n/N}, \quad h=(\log{T})^2.
\end{align*}
For any $t$ satisfying 
$$h\le t \le T,$$
we have 
\begin{align*}
\sum_{n=0}^{\infty}c_n n^{-\sigma-it}\ll N^{o(1)}\left(\int_{-h}^{h}\left|\zeta\left(\sigma+i(t+\tau)\right) \right|d\tau+1\right).
\end{align*}
\end{lemma}
\begin{lemma}
\label{lem:mainvlarge1}
Let $N,T\ge 1$ and $a_n$ a sequence of complex numbers satisfying $|a_n|\le 1.$ Let $\cA\subseteq [0,T]$ be a well spaced set satisfying 
\begin{align*}
\left|\sum_{N\le n\le 2N}a_n n^{it}\right|\ge V, \quad t\in \cA.
\end{align*}
Let $0<\delta<1$ and suppose $N,T,V$ are positive numbers with $N$ and $V$  satisfying 
\begin{align}
\label{eq:NTVcond}
V^{4}\ge N^{3+o(1)}.
\end{align}
We have 
\begin{align*}
& I(\delta T,\cA) \ll  \frac{N^{2+o(1)}}{V^2}|\cA|+\frac{N^{3/2+o(1)}}{V^2}\sum_{\substack{t_1,t_2\in \cA \\ |t_1-t_2|\le \delta T}}\left|\sum_{n\le \delta T/N}n^{-1/2+i(t_1-t_2)}\right|.
\end{align*}
\end{lemma}
\begin{proof}
For $0\le k \ll \delta^{-1}$ define
$$\cA_k=\cA\cap [k \delta T,(k+1)\delta T].$$
We have 
\begin{align*}
|\cA_k|V\le \sum_{t\in \cA_k}\left|\sum_{N\le n \le 2N}a_n n^{it}\right|\le N^{o(1)}\sum_{N\le n \le 2N}\left|\sum_{t\in \cA_k}\theta(t)n^{it}\right|,
\end{align*}
for some sequence of complex numbers $\theta$ satisfying $|\theta(t)|=1$. By the Cauchy-Schwarz inequality 
\begin{align}
\label{eq:Ak1}  \nonumber
|\cA_k|^2V^2&\ll N^{1+o(1)}\sum_{n=1}^{\infty}b(n)\left|\sum_{t\in \cA_k}\theta(t)n^{it}\right|^2 \\
 &\ll N^{1+o(1)}\sum_{t_1,t_2\in \cA_k}\left|\sum_{n=1}^{\infty}b(n)n^{i(t_1-t_2)}\right| 
\nonumber \\
&\ll N^{2+o(1)}I((\log{T})^2,\cA_k)+\sum_{\substack{t_1,t_2\in \cA_k \\ |t_1-t_2|\ge (\log{T})^2}}\left|\sum_{N\le n \le 2N}n^{i(t_1-t_2)}\right|,
\end{align}
with $b(n)$ is as in Lemma~\ref{lem:jut}. The assumption $\cA_k$ is well spaced implies 
\begin{align*}
I((\log{T})^2,\cA_k)\ll N^{o(1)}|\cA_k|,
\end{align*}
and for $t_1,t_2\in \cA_k$ satisfying $|t_1-t_2|\ge (\log{T})^2$ we have $|t_1-t_2|\le \delta T$, hence by Lemma~\ref{lem:jut}
\begin{align*}
\left|\sum_{N\le n \le 2N}n^{i(t_1-t_2)}\right|\ll N^{1/2}\int_{|\tau|\le (\log{T})^2}\left|\sum_{n\le \delta T/N}n^{-1/2+i(t_1-t_2+\tau)}\right|d\tau+1.
\end{align*}
Substituting the above into~\eqref{eq:Ak1} and taking a maximum over $\tau$, we get 
\begin{align*}
|\cA_k|^2V^2&\ll N^{3/2+o(1)}\sum_{\substack{t_1,t_2\in \cA_k}}\left|\sum_{n\le \delta T/N}a_nn^{-1/2+i(t_1-t_2)}\right| \\ & \quad \quad \quad +N^{2+o(1)}|\cA_k|+N^{3/2+o(1)}|\cA_k|^2,
\end{align*}
for some sequence of complex numbers $a_n$ satisfying $|a_n|=1$. By~\eqref{eq:NTVcond}, the above simplifies to 
\begin{align*}
|\cA_k|^2V^2&\ll N^{3/2+o(1)}\sum_{\substack{t_1,t_2\in \cA_k}}\left|\sum_{n\le \delta T/N}a_nn^{-1/2+i(t_1-t_2)}\right|+N^{2+o(1)}|\cA_k|.
\end{align*}
Summing over $|k|\ll \delta^{-1}$ and using Lemma~\ref{lem:ell2counting} to estimate 
$$\sum_{k\ll \delta^{-1}}|\cA_k|^2\gg I(\delta T,\cA),$$
we get 
\begin{align*}
I(\delta T,\cA)V^2 &\ll N^{2+o(1)}|\cA|+N^{3/2+o(1)}\sum_{k\ll \delta^{-1}}\sum_{\substack{t_1,t_2\in \cA_k}}\left|\sum_{n\le \delta T/N}a_nn^{-1/2+i(t_1-t_2)}\right| \\
&\ll N^{2+o(1)}|\cA|+N^{3/2+o(1)}\sum_{\substack{\substack{t_1,t_2\in \cA \\ |t_1-t_2|\le \delta T}}}\left|\sum_{n\le \delta T/N}a_nn^{-1/2+i(t_1-t_2)}\right|, \nonumber
\end{align*}
after noting that if $t_1,t_2\in \cA_k$ then $|t_1-t_2|\le \delta T$.
\end{proof}

\section{Zero density estimates}
We next collect some preliminaries from the method of zero detection polynomials, we refer the reader to~\cite[Chapter~11]{Ivic} or ~\cite[Chapter~12]{Mont}  for details.
\begin{lemma}
\label{lem:zerodensity}
Let $X,Y$ be parameters satisfying $2\le X \le Y \le T^A$ for some absolute constant $A$ and define
\begin{align*}
M_X(s)=\sum_{n\le X}\frac{\mu(n)}{n^s},
\end{align*} 
and
\begin{align*}
a_n=\sum_{\substack{d|n \\ d\le X}}\mu(d).
\end{align*}
There exists two well spaced sets $\cA_1,\cA_2 \subset \C$ such that
$$N(\sigma,T)\ll T^{o(1)}(|\cA_1|+|\cA_2|).$$
If $\rho=\beta+i\gamma \in \cA_1$ then 
$$\beta \ge \sigma, \quad  \gamma\le T,$$
and 
\begin{align}
\label{eq:zeroclass1}
\sum_{X\le n \le Y^2}a_n n^{-\rho}e^{-n/Y}\gg 1.
\end{align}
If $\rho=\beta+i\gamma \in \cA_2$ then
$$\beta \ge \sigma, \quad  \gamma\le T,$$
and 
\begin{align}
\label{eq:zeroclass2}
\int_{-C\log{T}}^{C\log{T}}\zeta\left(\frac{1}{2}+i(\gamma+\tau)\right)M_X\left(\frac{1}{2}+i(\gamma+\tau)\right)Y^{1/2-\beta+i\tau}& \Gamma\left(\frac{1}{2}-\beta +i\tau\right)d\tau  \\ & \quad \quad \quad \gg 1, \nonumber
\end{align}
where $C$ is some absolute constant.
\end{lemma}
Using some Fourier analysis one may remove the dependence of the coefficients in~\eqref{eq:zeroclass1} on the real part of the zeros $\rho$.
\begin{lemma}
\label{lem:zerodensity1}
Let $\sigma \ge 1/2+o(1)$. With notation as in Lemma~\ref{lem:zerodensity}, let $X,Y$ be parameters satisfying $2\le X \le Y \le T^A$. There exists  some $N$ satisfying
\begin{align}
\label{lem:density1N}
X\le N \le 2Y,
\end{align}
a sequence of complex numbers $b_n$ satisfying $|b_n|\le 1$ and two well spaced sets $\cA_1,\cA_2 \subset [0,T]$ satisfying
$$N(\sigma,T)\ll T^{o(1)}(|\cA_1|+|\cA_2|).$$
If $t\in \cA_1$ we have
\begin{align}
\label{eq:zeroclass1}
N^{o(1)}\sum_{N\le n \le 2N}b_n n^{it}\gg N^{\sigma}, \quad t\in \cA_1,
\end{align}
and if $t\in \cA_2$
\begin{align}
\label{eq:zeroclass2}
N^{o(1)}\zeta\left(\frac{1}{2}+it \right)M_X\left(\frac{1}{2}+it \right) \gg Y^{\sigma -1/2}, \quad t\in \cA_2.
\end{align}
\end{lemma}

We refer the reader to either~\cite[Section~1]{Bou} or~\cite[Chapter~11]{Ivic} for details of the following reduction from Lemma~\ref{lem:zerodensity1} which makes use of Heath-Brown's twelfth power moment estimate~\cite{HB0}.
\begin{lemma}
\label{lem:zerodensity2}
Let $Y\le T^{A}$ be some parameter. There exists some $N$ satisfying
\begin{align}
\label{eq:zerodensity1X}
Y^{4/3}<N<Y^{2+o(1)},
\end{align}
and some sequence of complex numbers $a_n$ satisfying $|a_n|\le 1$ such that for some well spaced set $\cA\subseteq [0,T]$ with
\begin{align*}
N^{o(1)}\left|\sum_{N\le n \le 2N}a_n n^{it}\right|\ge N^{\sigma} \quad t\in \cA,
\end{align*}
we have 
\begin{align*}
N(\sigma,T)\ll T^{o(1)}\left(|\cA|+T^2Y^{6(1-2\sigma)}\right).
\end{align*}
\end{lemma}

\section{Proof of Theorem~\ref{thm:main1}}
By Lemma~\ref{lem:mainvlarge1} we have 
\begin{align}
\label{eq:main1case22}
I(\delta T,\cA)\ll \frac{N^{2+o(1)}}{V^2}|\cA|+ \frac{N^{3/2+o(1)}}{V^2}W,
\end{align}
where
\begin{align*}
W=\sum_{\substack{t_1,t_2\in \cA \\ |t_1-t_2|\le \delta T}}\left|\sum_{n\le \delta T/N}n^{-1/2+i(t_1-t_2)}\right|.
\end{align*}
For integer $\ell$, define
\begin{align*}
r(\ell)=|\{ (t_1,t_2)\in \cA\times \cA \ : \ \ell< t_1-t_2\le \ell+1\}|,
\end{align*}
so that 
\begin{align*}
W\le \sum_{|\ell|\le \delta T}r(\ell)\max_{0\le \theta \le 1}\left|\sum_{n\le \delta T/N} n^{-1/2+i(\ell+\theta)}\right|.
\end{align*}
For integers $i,j\ge 0$ define the set 
\begin{align*}
D_{i,j}=\left\{ |\ell|\le \delta T \ : \ 2^i\le r(\ell)< 2^{i+1}, \ 2^j\le \max_{0\le \theta\le 1}\left|\sum_{n\le \delta T/N} n^{-1/2+i(\ell+\theta)}\right|< 2^{j+1} \right\},
\end{align*}
so that 
\begin{align}
\label{eq:WW0111}
W\ll W_0+I(\delta T,\cA),
\end{align}
where the factor $I(\delta T,\cA)$ comes from values of $\ell$ satisfying 
\begin{align*}
\max_{0\le \theta\le 1}\left|\sum_{n\le \delta T/N} n^{-1/2+i(\ell+\theta)}\right|\le 1,
\end{align*}
and $W_0$ is defined by 
\begin{align}
\label{eq:1W0def}
W_0=\sum_{i,j\ll \log{N}}\sum_{\ell \in D_{i,j}}r(\ell)\max_{0\le \theta \le 1}\left|\sum_{n\le N^2/\delta T} n^{-1/2+i(\ell+\theta)}\right|^2.
\end{align}
Note that for each $0\le i,j\ll \log{N}$ we have 
\begin{align*}
2^{i+j}|D_{i,j}|\ll\sum_{\ell \in D_{i,j}}r(\ell)\max_{0\le \theta \le 1}\left|\sum_{n\le \delta T/N} n^{-1/2+i(\ell+\theta)}\right|\ll 2^{i+j}|D_{i,j}|,
\end{align*}
and hence by the pigeonhole principle applied to the set $$\{(i,j) \ : 0\le i,j\ll \log{N}\},$$ there exists some pair $i,j$ satisfying 
\begin{align*}
2^{i+j}|D_{i,j}|\ll W_0\ll N^{o(1)}2^{i+j}|D_{i,j}|.
\end{align*}
Write 
\begin{align*}
 D=D_{i,j}, \quad \Delta=2^{i}, \quad H=2^j,
\end{align*}
so that 
\begin{align}
\label{eq:W0DH}
\Delta H |D|\ll W_0\ll N^{o(1)}\Delta  H |D|.
\end{align}
We consider two cases depending on 
\begin{align}
\label{eq:thmmain1case1}
|D|\ge N
\end{align}
or 
\begin{align}
\label{eq:thmmain1case2}
|D|<N.
\end{align}
Consider first~\eqref{eq:thmmain1case1}. From the definition of $H,D$, we have 
\begin{align*}
H^{2k}|D|&\ll \sum_{\ell \ll \delta T}\max_{0\le \theta \le 1} \left|\sum_{n\le \delta T/N}n^{-1/2+i(\ell+\theta)}\right|^{2k}
\end{align*}
and hence by~\eqref{eq:main1delta} and Lemma~\ref{lem:classicalmoments}
\begin{align}
\label{eq:H2kD}
H^{2k}|D|\ll (\delta T)^{1+o(1)}.
\end{align}
By~\eqref{eq:thmmain1case1} this implies 
\begin{align*}
H\ll \left(\frac{\delta T}{N}\right)^{1/2k}N^{o(1)}.
\end{align*}
From~\eqref{eq:W0DH} and the bound
\begin{align}
\label{eq:DeltaDthm1}
\Delta |D|\ll I(\delta T,\cA),
\end{align}
we get 
\begin{align*}
W_0\ll \left(\frac{\delta T}{N}\right)^{1/2k}I(\delta T,\cA)N^{o(1)}.
\end{align*}
By~\eqref{eq:main1case22}, ~\eqref{eq:WW0111} and~\eqref{eq:1W0def}
\begin{align*}
I(\delta T,\cA)\ll \frac{N^{2+o(1)}}{V^2}|\cA|+\frac{N^{3/2+o(1)}}{V^2}\left(\frac{\delta T}{N}\right)^{1/2k}I(\delta T,\cA),
\end{align*}
and hence from~\eqref{eq:main1delta}
\begin{align}
\label{eq:main1case1final}
I(\delta T,\cA)\ll \frac{N^{2+o(1)}}{V^2}.
\end{align}
Suppose next~\eqref{eq:thmmain1case2} and consider 
\begin{align*}
S=\int_{|\tau|\le 2\log{N}}\sum_{\substack{\ell \in D, t\in \cA \\ }}\left|\sum_{N\le n \le 2N}a_n n^{-i(\ell+t+\tau)}\right|^2d\tau.
\end{align*} 
Define the set $\cB$ by 
$$\cB=\{ (\ell,t)\in D\times \cA \ : \exists \ t'\in \cA \ \text{such that} \ |t'-t-\ell|\le 2\},$$
so that 
\begin{align}
\label{eq:SubB}
S\ge \sum_{\substack{(\ell,t)\in \cB }}\int_{|\tau|\le 2\log{N}}\left|\sum_{N\le n \le 2N}a_n n^{i(\ell+t+\tau)}\right|^2d\tau.
\end{align}
If $(\ell,t)\in \cB$ then there exists some $t'\in \cA$ and some $|\theta|\le 2$ such that 
$$t'+\theta=\ell+t,$$
and hence 
\begin{align*}
\int_{|\tau|\le 2\log{N}}\left|\sum_{N\le n \le 2N}a_n n^{i(\ell+t+\tau)}\right|^2d\tau=\int_{|\tau|\le 2\log{N}}\left|\sum_{N\le n \le 2N}a_n n^{i(t'+\theta+\tau)}\right|^2d\tau.
\end{align*}
Since $|\theta|\le 2$ we have 
$$(\theta+[-2\log{N},2\log{N}])\cap [-\log{N},\log{N}]=[-\log{N},\log{N}],$$
and hence 
\begin{align*}
\int_{|\tau|\le 2\log{N}}\left|\sum_{N\le n \le 2N}a_n n^{i(\ell+t+\tau)}\right|^2d\tau\ge \int_{|\tau|\le \log{N}}\left|\sum_{N\le n \le 2N}a_n n^{i(t'+\tau)}\right|^2d\tau.
\end{align*}
Applying Lemma~\ref{lem:removemax} gives 
\begin{align*}
\int_{|\tau|\le 2\log{N}}\left|\sum_{N\le n \le 2N}a_n n^{i(\ell+t+\tau)}\right|^2d\tau\gg \frac{V^2}{\log{N}},
\end{align*}
and hence by~\eqref{eq:SubB}
\begin{align*}
N^{o(1)}S\gg V^2|\cB|.
\end{align*}
Recalling the definition of $D,\Delta$, we have 
\begin{align*}
\Delta |D|\ll \sum_{\ell \in D}r(\ell)\ll |\{ (\ell,t,t')\in D \times \cA \times \cA \ : \ 0<t'-t-\ell \le 1\}|\le |\cB|,
\end{align*}
which implies 
\begin{align}
\label{eq:SdeltaDLB}
\Delta |D|\ll \frac{N^{o(1)}}{V^2}S.
\end{align}
Taking a maximum over $\tau$ in $S$, there exists some sequence of complex numbers $b(n)$ satisfying $|b(n)|\le 1$ such that 
\begin{align*}
S&\ll N^{o(1)}\sum_{\ell\in D, t\in \cB}\left|\sum_{N\le n \le 2N}b(n)n^{i(\ell+t)}\right|^2 \\ 
&\le  \sum_{N\le n_1,n_2\le 2N}\left|\sum_{\ell\in D}\left(\frac{n_1}{n_2}\right)^{i\ell} \right|\left|\sum_{t\in \cB}\left(\frac{n_1}{n_2}\right)^{it}\right|.
\end{align*}
By the Cauchy-Schwarz inequality
\begin{align}
\label{eq:SS1S2}
S^2\ll N^{2+o(1)}S_1S_2,
\end{align}
where 
\begin{align*}
S_1=\sum_{N\le n_1,n_2\le 2N}(n_1n_2)^{-1/2}\left|\sum_{\ell\in D}\left(\frac{n_1}{n_2}\right)^{i\ell} \right|^2,
\end{align*}
and 
\begin{align*}
S_2=\sum_{N\le n_1,n_2\le 2N}(n_1n_2)^{-1/2}\left|\sum_{t\in \cA}\left(\frac{n_1}{n_2}\right)^{it} \right|^2.
\end{align*}
Interchanging summation, we have 
\begin{align*}
S_1\ll \sum_{\ell_1,\ell_2\in D}\left|\sum_{N\le n \le 2N}n^{-1/2+i(\ell_1-\ell_2)}\right|^2.
\end{align*}
and hence by Theorem~\ref{thm:heathbrown}  and~\eqref{eq:thmmain1case2}
\begin{align*}
S_1\ll N^{1+o(1)}|D|.
\end{align*}
By a similar argument 
\begin{align*}
S_2\ll N^{1+o(1)}|\cA|.
\end{align*}
From the above,~\eqref{eq:SdeltaDLB} and~\eqref{eq:SS1S2}
\begin{align*}
\Delta^2|D|\ll \frac{N^{4+o(1)}}{V^4}|\cA|.
\end{align*}
Returning to~\eqref{eq:W0DH}, we have 
\begin{align*}
W_0\ll N^{o(1)}(\Delta^2 |D|)^{1/2k}(\Delta |D|)^{1-1/k}(|D|H^{2k})^{1/2k}.
\end{align*}
The above,~\eqref{eq:H2kD} and~\eqref{eq:DeltaDthm1} imply 
\begin{align*}
W_0\ll \left(\frac{N^{4+o(1)}}{V^4}|\cA|\right)^{1/2k}I(\delta T,\cA)^{1-1/k}(\delta T)^{1/2k}.
\end{align*}
By~\eqref{eq:main1case22} and~\eqref{eq:WW0111}
\begin{align*}
I(\delta T,\cA)\ll \frac{N^{2+o(1)}}{V^2}|\cA|+\frac{N^{3k/2+2+o(1)}}{V^{2k+2}}|\cA|^{1/2}(\delta T)^{1/2}.
\end{align*}
From~\eqref{eq:main1case1final} the above bound holds provided either~\eqref{eq:thmmain1case1} or~\eqref{eq:thmmain1case2}. An application of Lemma~\ref{eq:ell2ell2c} implies 
\begin{align*}
|\cA|^2\ll \frac{1}{\delta}\frac{N^{2+o(1)}}{V^2}+\frac{1}{\delta^{1/3}}\frac{N^{k+4/3+o(1)}}{V^{4k/3+4/3}}T^{1/3},
\end{align*}
which completes the proof.

\section{Proof of Theorem~\ref{thm:main4}}

By Lemma~\ref{lem:mainvlarge1}
\begin{align*}
& I(\delta T,\cA) \ll  \frac{N^{2+o(1)}}{V^2}|\cA|+\frac{N^{3/2+o(1)}}{V^2}\sum_{\substack{t_1,t_2\in \cA \\ |t_1-t_2|\le \delta T}}\left|\sum_{n\le \delta T/N}n^{-1/2+i(t_1-t_2)}\right|,
\end{align*}
and H\"{o}lder's inequality 
\begin{align}
\label{eq:IDD33333}
I(\delta T,\cA)\ll \frac{N^{2+o(1)}}{V^2}|\cA|+\frac{N^{6+o(1)}}{V^{8}}W,
\end{align}
where 
\begin{align*}
W=\sum_{\substack{t_1,t_2\in \cA \\ |t_1-t_2|\le \delta T}}\left|\sum_{n\le \delta T/N}n^{-1/2+i(t_1-t_2)}\right|^4.
\end{align*}
We have 
\begin{align*}
W=\sum_{\substack{t_1,t_2\in \cA \\ |t_1-t_2|\le \delta T}}\left|\sum_{n\le (\delta T/N)^2}c_nn^{-1/2+i(t_1-t_2)}\right|^2,
\end{align*}
for some $c_n=N^{o(1)}$. By Lemma~\ref{lem:coefficients} and Corollary~\ref{cor:reflection}
\begin{align*}
W&\ll I(\delta T,\cA)N^{o(1)}+\frac{\delta^2 T^2N^{o(1)}}{N^2}|\cA|+\sum_{\substack{t_1,t_2\in \cA \\ |t_1-t_2|\le \delta T}}\left|\sum_{n\le 4N^2/(\delta T)}n^{-1/2+i(t_1-t_2)}\right|^2,
\end{align*}
which by~\eqref{eq:IDD33333} implies that 
\begin{align}
\label{eq:main4prelimstep}
I(\delta T,\cA)\ll \frac{N^{2+o(1)}}{V^2}|\cA|+\frac{\delta^2 T^2 N^{4+o(1)}}{V^{8}}|\cA|+\frac{N^{6+o(1)}}{V^8}W_0,
\end{align}
with 
\begin{align*}
W_0=\sum_{\substack{t_1,t_2\in \cA \\ |t_1-t_2|\le \delta T}}\left|\sum_{n\le 4N^2/(\delta T)}n^{-1/2+i(t_1-t_2)}\right|^2.
\end{align*}
This implies either 
\begin{align*}
|\cA|\ll \frac{1}{\delta}\frac{N^{2+o(1)}}{V^2}+\frac{\delta T^2 N^{4+o(1)}}{V^{8}},
\end{align*}
or 
\begin{align}
\label{eq:main4case222b}
I(\delta T,\cA)\ll \frac{N^{6+o(1)}}{V^8}W_0.
\end{align}
We may suppose~\eqref{eq:main4case222b} since otherwise the result follows. Considering $W_0$, partitioning summation over $n$ into dyadic intervals and applying Corollary~\ref{cor:larger} gives 
\begin{align*}
W_0\ll N^{o(1)}\sum_{\substack{t_1,t_2\in \cA \\ |t_1-t_2|\le \delta T}}\left|\sum_{16 N^2/(\delta T)\le n\le 32N^2/(\delta T)}n^{-1/2+i(t_1-t_2)}\right|^2,
\end{align*}
and hence by Lemma~\ref{lem:coefficients}
\begin{align}
\label{eq:W0111222main4}
W_0\ll \left(\frac{N^{2+o(1)}}{\delta T}\right)^{1/4}\sum_{\substack{t_1,t_2\in \cA \\ |t_1-t_2|\le \delta T}}\left|\sum_{16 N^2/(\delta T)\le n\le 32N^2/(\delta T)}n^{-5/8+i(t_1-t_2)}\right|^2.
\end{align}
Bounding the contribution from points $t_1,t_2$ satisfying $|t_1-t_2|\le N^{o(1)}$ trivially and applying Lemma~\ref{lem:jut1} to the remaining sum, we get
\begin{align}
\label{eq:W01main4}
W_0\ll \frac{N^{2+o(1)}}{\delta T}|\cA|+\left(\frac{N^{2+o(1)}}{\delta T}\right)^{1/4}I(\delta T,\cA)+\left(\frac{N^{2+o(1)}}{\delta T}\right)^{1/4}W_1,
\end{align}
where 
\begin{align*}
W_1=\sum_{\substack{t_1,t_2\in \cA \\ |t_1-t_2|\le \delta T}}\int_{-h^2}^{h^2}\left|\zeta\left(\frac{5}{8}+i(t_1-t_2+\tau)\right) \right|^2d\tau,
\end{align*}
\begin{align*}
h=(\log{T})^2,
\end{align*}
and we have used a second application of Lemma~\ref{lem:coefficients} to the sum~\eqref{eq:W0111222main4} in order to smooth the coefficients. Performing a dyadic partition as in the proof of Theorem~\ref{thm:main1}, there exists 
$$\Delta,H\gg 1,$$
and a set $D\subseteq \Z$ defined by 
\begin{align*}
&D= \\ & \left\{ |\ell|\le \delta T \ : \ \Delta \le r(\ell)< 2\Delta, \ H\le \int_{-2h^2}^{2h^2}\left|\zeta\left(\frac{5}{8}+i(\ell+\tau)\right) \right|^2d\tau< 2H \right\},
\end{align*}
such that
\begin{align}
\label{eq:WW0}
W_1\ll I(\delta T,\cA)+N^{o(1)}\Delta |D|H^2,
\end{align}
where 
\begin{align*}
r(\ell)=|\{ (t_1,t_2)\in \cA\times \cA \ : \ 0\le t_1-t_2-\ell <1 \}|.
\end{align*}
Note by Lemma~\ref{lem:8thmoment}
\begin{align}
\label{eq:D8}
|D|H^8\ll (\delta T)^{1+o(1)}.
\end{align}
Either 
\begin{align}
\label{eq:main4case1}
|D|\le N \quad \text{and} \quad |D|\le \frac{N^4}{(\delta T)^2},
\end{align}
or 
\begin{align}
\label{eq:main4case2}
|D|> N \quad \text{or} \quad |D|> \frac{N^4}{(\delta T)^2}.
\end{align}
If~\eqref{eq:main4case1}, then arguing as in the proof of Lemma~\ref{lem:e2energy},
\begin{align*}
V^2\Delta |D|\ll N^{o(1)}\sum_{d\in D,t\in \cA}\left|\sum_{N\le n \le 2N}c_n n^{i(t-d)}\right|^2,
\end{align*}
for some sequence of complex numbers $c_n$ satisfying $|c_n|\le 1$. Expanding the square, interchanging summation, applying the Cauchy-Schwarz inequality and rescaling we get 
\begin{align}
\label{eq:main42bb}
V^4\Delta^2|D|^2\ll N^{2+o(1)}\sum_{t_1,t_2\in \cA}\left|\sum_{N\le n \le 2N}n^{i(t_1-t_2)}\right|^2\sum_{d_1,d_2\in D}\left|\sum_{N\le n \le 2N}n^{i(d_1-d_2)}\right|^2.
\end{align}
By Theorem~\ref{thm:heathbrown}
\begin{align*}
\sum_{t_1,t_2\in \cA}\left|\sum_{N\le n \le 2N}n^{i(t_1-t_2)}\right|^2&\ll N^{o(1)}(|\cA|^2+N|\cA|+T^{1/2}|\cA|^{5/4}) \\
&\ll N^{1+o(1)}|\cA|,
\end{align*}
and 
\begin{align*}
\sum_{d_1,d_2\in \cA}\left|\sum_{N\le n \le 2N}n^{i(d_1-d_2)}\right|^2&\ll N^{o(1)}(|D|^2+N|D|+(\delta T)^{1/2}|D|^{5/4}) \\
&\ll N^{1+o(1)}|D|,
\end{align*}
where we have used~\eqref{eq:main4ass} and~\eqref{eq:main4case1} to simplify the above bounds. Combining with~\eqref{eq:main42bb} 
\begin{align*}
\Delta^2|D|\ll \frac{N^{4+o(1)}}{V^4}|\cA|.
\end{align*}
From the above and~\eqref{eq:D8}  
\begin{align*}
\Delta |D|H^2&\ll (\Delta^2|D|)^{1/4}(\Delta |D|)^{1/2}(|D|H^8)^{1/4} \\
&\ll \frac{N^{1+o(1)}}{V}|\cA|^{1/4}I(\delta T,\cA)^{1/2}(\delta T)^{1/4}.
\end{align*}
By~\eqref{eq:main4case222b},~\eqref{eq:W01main4} and~\eqref{eq:WW0}
\begin{align*}
I(\delta T,\cA)& \ll \frac{N^{8+o(1)}}{(\delta T)V^8}|\cA|+\frac{N^{13/2+o(1)}}{(\delta T)^{1/4}V^8}I(\delta T,\cA) \\ & \quad \quad \quad \quad +\frac{N^{13/2+o(1)}}{(\delta T)^{1/4}V^8}\frac{N^{1+o(1)}}{V}|\cA|^{1/4}I(\delta T,\cA)^{1/2}(\delta T)^{1/4} \\ 
&\ll \frac{N^{8+o(1)}}{(\delta T)V^8}|\cA|+\frac{N^{15/2+o(1)}}{V^9}|\cA|^{1/4}I(\delta T,\cA)^{1/2},
\end{align*}
where we have used~\eqref{eq:main4deltaass} to drop the second term. The above estimate implies that either 
\begin{align*}
|\cA|\ll \frac{N^{8+o(1)}}{\delta^2 TV^8},
\end{align*}
or 
\begin{align*}
I(\delta T,\cA)\ll \frac{N^{15+o(1)}}{V^{18}}|\cA|^{1/2},
\end{align*}
and hence 
\begin{align*}
|\cA|\ll \frac{N^{8+o(1)}}{\delta^2 TV^8}+\frac{N^{10+o(1)}}{\delta^{2/3}V^{12}},
\end{align*}
which completes the proof of case~\eqref{eq:main4case1}. Suppose next~\eqref{eq:main4case2}. From~\eqref{eq:D8} we have either 
\begin{align*}
H^2\ll \left(\frac{(\delta T)}{N}\right)^{1/4}N^{o(1)},
\end{align*}
or 
\begin{align*}
H^2 \ll \frac{(\delta T)^{3/4}}{N}N^{o(1)},
\end{align*}
which implies
\begin{align*}
H^2\ll \left(\frac{\delta T}{N}\right)^{1/4}N^{o(1)}+\frac{(\delta T)^{3/4}}{N}N^{o(1)}.
\end{align*}
By~\eqref{eq:WW0}
\begin{align}
\label{eq:WW0}
W_1\ll N^{o(1)}\left(1+\left(\frac{\delta T}{N}\right)^{1/4}+\frac{(\delta T)^{3/4}}{N}\right)I(\delta T,\cA),
\end{align}
and hence by~\eqref{eq:main4case222b} and~\eqref{eq:W01main4}
\begin{align*}
I(\delta T,\cA) & \ll \frac{N^{8+o(1)}}{\delta TV^8}|\cA|
\\& \quad \quad +N^{o(1)}\left(\frac{N^{13/2}}{V^8(\delta T)^{1/4}}+\frac{N^{25/4}}{V^8}+\frac{N^{11/2}(\delta T)^{1/2}}{V^8}\right)I(\delta T,\cA),
\end{align*}
By~\eqref{eq:main4ass} and~\eqref{eq:main4deltaass}  we may drop the last three terms to arrive at 
\begin{align*}
|\cA|\ll \frac{N^{8+o(1)}}{\delta^2 TV^8},
\end{align*}
and the result follows combining the estimates from cases~\eqref{eq:main4case1} and~\eqref{eq:main4case2}.
\section{Proof of Theorem~\ref{thm:main12}}
Let $0<\delta\le 1$ be some parameter satisfying  
\begin{align}
\label{eq:thm12deltadef}
N^{o(1)}\delta \le \frac{V^8}{TN^5},
\end{align}
and apply Lemma~\ref{lem:mainvlarge1} and the Cauchy-Schwarz inequality to get 
\begin{align*}
& I(\delta T,\cA) \ll  \frac{N^{2+o(1)}}{V^2}|\cA|+\frac{N^{3+o(1)}}{V^4}S,
\end{align*}
where 
\begin{align*}
S=\sum_{\substack{\substack{t_1,t_2\in \cA \\ |t_1-t_2|\le \delta T}}}\left|\sum_{n\le \delta T/N}a_nn^{-1/2+i(t_1-t_2)}\right|^2.
\end{align*}
This implies that either 
\begin{align}
\label{eq:main12case1}
I(\delta T,\cA) \ll  \frac{N^{2+o(1)}}{V^2}|\cA|,
\end{align}
or 
\begin{align}
\label{eq:main12case2}
& I(\delta T,\cA) \ll  \frac{N^{3+o(1)}}{V^4}S.
\end{align}
If~\eqref{eq:main12case1} then by Corollary~\ref{eq:ell2ell2c}
\begin{align}
\label{eq:main12case1z}
|\cA|\ll \frac{1}{\delta}\frac{N^{2+o(1)}}{V^2}.
\end{align}
Suppose next~\eqref{eq:main12case2}. For integer $\ell$ define 
\begin{align*}
r(\ell)=|\{ (t_1,t_2)\in \cA\times \cA \ : \ 0\le t_1-t_2-\ell<1\},
\end{align*}
so that 
\begin{align*}
S\le \sum_{\substack{\substack{ |\ell|\le \delta T}}}r(\ell)\max_{0\le \theta \le 1}\left|\sum_{n\le \delta T/N}a_nn^{-1/2+i(\ell+\theta)}\right|^2.
\end{align*}
Using Lemma~\ref{lem:removemax} to remove the maximum over $\theta$, there exists a sequence of complex numbers $c_n$ satisfying $|c_n|\le 1$ such that 
\begin{align*}
S\ll N^{o(1)}\sum_{\substack{\substack{|\ell| \le \delta T}}}r(\ell)\left|\sum_{n\le \delta T/N}c_nn^{-1/2+i\ell}\right|^2.
\end{align*}
Expanding the square and interchanging summation 
\begin{align*}
S\ll N^{o(1)}\sum_{n_1,n_2\le \delta T/N}(n_1n_2)^{-1/2}\left|\sum_{|\ell| \le \delta T}r(\ell)\left(\frac{n_1}{n_2}\right)^{i\ell} \right|,
\end{align*}
and hence by the Cauchy-Schwarz inequality 
\begin{align*}
S^2\ll N^{o(1)}\left(\frac{\delta T}{N}\right)\sum_{n_1,n_2\le \delta T/N}(n_1n_2)^{-1/2}\left|\sum_{|\ell| \le \delta T}r(\ell)\left(\frac{n_1}{n_2}\right)^{i\ell} \right|^2,
\end{align*}
which rearranges to 
\begin{align*}
S^2\ll N^{o(1)}\left(\frac{\delta T}{N}\right)\sum_{|\ell_1|,|\ell_2|\le \delta T}r(\ell_1)(\ell_2)\left|\sum_{n\le \delta T/N}n^{-1/2+i(\ell_1-\ell_2)}\right|^2.
\end{align*}
Substituting into~\eqref{eq:main12case2} we get 
\begin{align*}
I(\delta T,\cA)^2\ll \frac{\delta TN^{5+o(1)}}{V^8}\sum_{|\ell_1|,|\ell_2|\le \delta T}r(\ell_1)r(\ell_2)\left|\sum_{n\le \delta T/N}n^{-1/2+i(\ell_1-\ell_2)}\right|^2.
\end{align*}
By Theorem~\ref{thm:largeadditive} 
\begin{align*}
&\sum_{|\ell_1|,|\ell_2|\le \delta T}r(\ell_1)r(\ell_2)\left|\sum_{n\le \delta T/N}n^{-1/2+i(\ell_1-\ell_2)} \right|^2
\\ & \ll  N^{o(1)}\left(\|r\|^2_1+\frac{\delta T}{N}\|r\|^2_2+(\delta T)^{1/2}\|r\|^{1/2}_1\|r\|^{3/2}_2\right),
 \end{align*}
where 
\begin{align}
\label{eq:rr121}
\|r\|_1=\sum_{|\ell| \le  \delta T}r(\ell), \quad \text{and} \quad  \|r\|^2_2=\sum_{|\ell| \le \delta T}r(\ell)^2,
\end{align}
so that
\begin{align*}
\|r\|_1\ll I(\delta T,\cA).
\end{align*}
By the above and~\eqref{eq:thm12deltadef}
\begin{align*}
I(\delta T,\cA)^2\ll \frac{\delta TN^{5+o(1)}}{V^8}\left(\frac{\delta T}{N}\|r\|^2_2+(\delta T)^{1/2}\|r\|^{1/2}_1\|r\|^{3/2}_2 \right),
\end{align*}
which implies 
\begin{align*}
I(\delta T,\cA)^2\ll \frac{(\delta T)^2N^{10+o(1)}}{V^{16}}\|r\|^2_2.
\end{align*}
By~\eqref{eq:rr121} and Lemma~\ref{lem:e2energy}
\begin{align*}
\|r\|_2^2\ll \frac{N^{3/2+o(1)}}{V^2}|\cA|^{1/2}I(\delta T,\cA)+\frac{N^{4+o(1)}}{V^{4}}|\cA|,
\end{align*}
so that 
\begin{align*}
I(\delta T,\cA)^2\ll \frac{(\delta T)^2N^{10+o(1)}}{V^{16}}\left(\frac{N^{3/2+o(1)}}{V^2}|\cA|^{1/2}I(\delta T,\cA)+\frac{N^{4+o(1)}}{V^{4}}|\cA| \right).
\end{align*}
The above simplifies to 
\begin{align*}
I(\delta T,\cA)\ll \left(\frac{(\delta T)^2N^{23/2}}{V^{18}}+\frac{(\delta T)N^7}{V^{10}}\right)N^{o(1)}|\cA|^{1/2},
\end{align*}
and hence by Corollary~\ref{eq:ell2ell2c}
\begin{align*}
|\cA|\ll \frac{\delta^{2/3} T^{4/3}N^{23/3+o(1)}}{V^{12}}+\frac{T^{2/3}N^{14/3+o(1)}}{V^{20/3}}.
\end{align*}
The result follows combining the above  with~\eqref{eq:main12case1z}.
\section{Proof of Theorem~\ref{thm:zerodensity2}}
We apply Lemma~\ref{lem:zerodensity2} with 
\begin{align*}
Y=T^{3/(8\sigma)-o(1)},
\end{align*}
to get
\begin{align}
\label{eq:density1s1q}
N(\sigma,T)\ll T^{o(1)}\left(|\cA|+T^2Y^{6(1-2\sigma)}\right)\ll T^{o(1)}\left(|\cA|+T^{3(1-\sigma)/2\sigma}\right),
\end{align}
where the set $\cA$ is a well spaced set $\cA\subseteq [0,T]$ satisfying
\begin{align*}
N^{o(1)}\left|\sum_{N\le n \le 2N}a_n n^{it}\right|\ge N^{\sigma} \quad t\in \cA,
\end{align*}
for some sequence of complex numbers $a_n$ with $|a_n|\le 1$ and $N$ satisfying
\begin{align}
\label{eq:Ncondthm119}
Y^{4/3}\le N \le Y^{2+o(1)}.
\end{align}
For $N$ satisfying~\eqref{eq:Ncondthm119} and $\sigma$ satisfying~\eqref{eq:sigmacond3}  the conditions of Theorem~\ref{thm:main4} are satisfied with the choice
\begin{align*}
N^{o(1)}\delta=\min\left\{1,\frac{V^3}{TN}\right\}.
\end{align*}  
This gives 
\begin{align}
\label{eq:main4zerodensity}
|\cA|&\ll N^{2(1-\sigma)+o(1)}+TN^{3-5\sigma+o(1)}+TN^{10-14\sigma+o(1)}+T^{2/3}N^{32/3-14\sigma+o(1)} 
      \nonumber  \\ & \quad \quad \quad +\frac{N^{8-8\sigma+o(1)}}{T}+N^{10-12\sigma+o(1)}.
\end{align}
We also note by Lemma~\ref{lem:hux}
\begin{align}
\label{eq:jutilazerodensity}
|\cA|\ll N^{2(1-\sigma)+o(1)}+TN^{4-6\sigma+o(1)}.
\end{align}
Let 
$$\rho=\min\left\{\frac{3(1-\sigma)}{2\sigma(10-12\sigma)},\frac{3}{16\sigma}+\frac{1}{8(1-\sigma)}\right\},$$
and apply~\eqref{eq:main4zerodensity} if 
$$Y^{4/3}\le N \le T^{\rho},$$
while if   
$$T^{\rho}\le N \le Y^2,$$
then apply~\eqref{eq:jutilazerodensity}. Provided $\sigma$ satisfies~\eqref{eq:sigmacond3} this gives the bound 
\begin{align*}
|\cA|\ll T^{3(1-\sigma)/2\sigma+o(1)},
\end{align*}
and the result follows from~\eqref{eq:density1s1q}.
\section{Proof of Theorem~\ref{thm:zerodensity1}}
Let $Y$ be some parameter to be determined later satisfying
\begin{align}
\label{eq:Ycond1}
Y\ge T^{1/2},
\end{align}
and apply Lemma~\eqref{lem:zerodensity2} to get
\begin{align}
\label{eq:density1s1}
N(\sigma,T)\ll T^{o(1)}\left(|\cA|+T^2Y^{6(1-2\sigma)}\right)\ll T^{o(1)}\left(|\cA|+T^{5-6\sigma}\right)
\end{align}
where the set $\cA$ is a well spaced set $\cA\subseteq [0,T]$ satisfying
\begin{align*}
N^{o(1)}\left|\sum_{N\le n \le 2N}a_n n^{it}\right|\ge N^{\sigma} \quad t\in \cA,
\end{align*}
for some sequence of complex numbers $a_n$ satisfying $|a_n|\le 1$ and $N$ satisfying
\begin{align*}
Y^{4/3}\le N \le Y^{2+o(1)}.
\end{align*}
Note by~\eqref{eq:Ycond1}
$$Y^{4/3}\ge T^{2/3}.$$
Applying Theorem~\ref{thm:main1} with $k=7$ and 
$$N^{o(1)}\delta=\min\left\{\frac{N^{7/6}}{T},1\right\},$$
gives
\begin{align}
\label{eq:zzthm1}
|\cA| &\ll  N^{2(1-\sigma)+o(1)}+T^{1/3}N^{25/3-32\sigma/3+o(1)} \\ 
&+TN^{5/6-2\sigma+o(1)}+T^{2/3}N^{143/18-32\sigma/3+o(1)}, \nonumber
\end{align}
where we have used the fact that the assumption~\eqref{eq:density1cond} implies 
\begin{align*}
28\sigma-20\ge \frac{7}{6}.
\end{align*}
We also note by Lemma~\ref{lem:hux}
\begin{align}
\label{eq:zzttc}
|\cA|\ll N^{2(1-\sigma)+o(1)}+TN^{4-6\sigma+o(1)}.
\end{align}
Let 
\begin{align}
\label{eq:Zcond111}
Y^{4/3}\le Z\le Y^2,
\end{align}
 be some parameter and apply~\eqref{eq:zzthm1} if
 $Y^{4/3}\le N\le T$ while apply~\eqref{eq:zzttc}  if $Z\le N\le Y^2$. This gives 
\begin{align*}
& |\cA| \ll T^{1/3}Z^{25/3-32\sigma/3+o(1)}+TZ^{4-6\sigma+o(1)} \\ 
&+TY^{4(5/6-2\sigma)/3+o(1)}+T^{2/3}Y^{4(143/18-32\sigma/3)/3+o(1)}+Y^{4(1-\sigma)+o(1)}.
\end{align*}
Choosing 
$$Y=T^{9/(138\sigma-89)}, \quad Z=T^{2/(13-14\sigma)},$$
to balance appropriate terms gives
\begin{align*}
|\cA|& \ll T^{(21-26\sigma)/(13-14\sigma)+o(1)}+T^{36(1-\sigma)/(138\sigma-89)+o(1)}+T^{(114\sigma-79)/(138\sigma-89)+o(1)} \\
&\ll T^{36(1-\sigma)/(138\sigma-89)+o(1)}+T^{(114\sigma-79)/(138\sigma-89)+o(1)},
\end{align*}
by~\eqref{eq:density1cond}.
Combining with~\eqref{eq:density1s1} we get  
\begin{align*}
N(\sigma,T)\ll T^{36(1-\sigma)/(138\sigma-89)+o(1)}+T^{(114\sigma-79)/(138\sigma-89)+o(1)}.
\end{align*}
It remains to note the conditions~\eqref{eq:Ycond1} and~\eqref{eq:Zcond111} are satisfied from~\eqref{eq:density1cond}.

\end{document}